\documentclass{amsart}

\usepackage{amsmath,amssymb,amsfonts,amsthm,amscd,indentfirst}

\usepackage[dvipdfm]{hyperref}
\begin{document}

\def\hom{\mathop{\cH om\skp}}
\def\ext{\mathop{\cE xt\skp}}

\newtheorem{prop}{Proposition}[section]
\newtheorem{theo}{Theorem}[section]
\newtheorem{lemm}{Lemma}[section]
\newtheorem{coro}{Corollary}[section]
\newtheorem{rema}{Remark}[section]
\newtheorem{exam}{Example}[section]
\newtheorem{defi}{Definition}[section]
\newtheorem{conj}{Conjecture}[section]
%\newcommand{\pf}{{\em Proof. }}

%\numberwithin{equation}{section}

\def\begeq{\begin{equation}}
\def\endeq{\end{equation}}
\def\begarr{\begin{array}{rll}}
\def\endarr{\end{array}}
\def\earr{\begin{equation}\begin{array}{rll}}
\def\eearr{\end{array}\end{equation}}

\def\and{\quad{\rm and}\quad}
\def\dd{\bf $\diamond$}

\def\U{\mathcal{U}}
\def\V{\mathcal{V}}
\def\W{\mathcal{W}}

\def\bl{\bigl(}
\def\br{\bigr)}
\def\lbe{_{\beta}}
\def\bN{{\mathbf N}}
\def\bs{{\mathbf s}}
\def\dist{{\mathbf dist}}
\def\<{\langle}
\def\>{\rangle}

\def\clo{\mathop{\rm cl\skp}}
\def\cdott{\!\cdot\!}
\def\cu{^{{\rm cu}}}

\def\lab{\label}

\def\Dint{\displaystyle\int}
\def\Dfrac{\displaystyle\frac}
\def\a{\alpha}
\def\b{\beta}
\def\r{\gamma}
\def\di{\displaystyle}

\title[Differential Harnack inequalities]{ Differential Harnack inequalities on Riemannian manifolds I : linear heat equation}

\author{Junfang Li}
\address{Department of Mathematics\\
         University of Alabama at Birmingham\\
         Birmingham, AL 35294, USA.}
\email{jli@math.uab.edu}
\author{Xiangjin Xu }
\address{Department of Mathematical Sciences\\
         Binghamton University\\
         Binghamton, New York, 13902, USA.}
\email{xxu@math.binghamton.edu}

%\date{07-07-2008}
\begin{abstract}
In the first part of this paper, we get new Li-Yau type gradient
estimates for positive solutions of heat equation on Riemmannian
manifolds with $Ricci(M)\ge -k$, $k\in \mathbb R$. As applications,
several parabolic Harnack inequalities are obtained and they lead to
new estimates on heat kernels of manifolds with Ricci curvature
bounded from below. In the second part, we establish a Perelman type
Li-Yau-Hamilton differential Harnack inequality for heat kernels on
manifolds with $Ricci(M)\ge -k$, which generalizes a result of L. Ni
\cite{NL1,NL4}. As applications, we obtain new Harnack inequalities and
heat kernel estimates on general manifolds. We also obtain various entropy monotonicity formulas for all compact Riemannian manifolds.
\end{abstract}
\thanks{Research of the second author was supported in part by the NSF grant DMS-0602151 and DMS-0852507.}

\maketitle
\tableofcontents
\section{\bf Introduction and main results}
Let $(M^n,g)$ be a complete Rimannian manifold. In the fundamental paper \cite{LY}, Li and Yau studied the heat equation solutions
\begeq\label{heat equ}
\partial_t u=\Delta_{g} u
\endeq
on general Riemannian manifolds. The results in \cite{LY} has tremendous impact in the field of geometric analysis. One of the fundamental results is the following important gradient estimates for heat equations.\\

\noindent{\bf Theorem} \label{LY}{(\bf Li-Yau \cite{LY})}\quad{\em
Let $(M,g)$ be a complete Riemannian manifold. Assume that on the ball $B_{2R}$, $Ric(M)\ge -k$. Then for any $\alpha>1$, we have that
\earr\label{LY equ1}
\displaystyle\sup_{B_R}\Big( \frac{|\nabla u|^2}{u^2} - \alpha\frac{u_t}{u} \Big)\le&\frac{C\alpha^2}{R^2}\Big(\frac{\alpha^2}{\alpha^2-1}+\sqrt k R \Big)+\frac{n\alpha^2k}{2(\alpha-1)}+\frac{n\alpha^2}{2t}.
\eearr
Moveover, when $(M,g)$ has nonnegative Ricci curvature, letting $R\rightarrow \infty$ and $\alpha\rightarrow 1$, (\ref{LY equ1}) gives the sharp estimate (a Hamilton-Jacobi inequality):
\earr\label{LY equ2}
 \frac{|\nabla u|^2}{u^2} - \frac{u_t}{u}\le\frac{n}{2t}.
\eearr
}

When $Ricci(M)\ge 0$, (\ref{LY equ2}) gives a clean sharp estimate.
In general, on a complete Riemannian manifold, if $Ricci(M)\ge -k$,
by letting $R\rightarrow \infty$ in (\ref{LY equ2}), one obtains
$$
\frac{|\nabla u|^2}{u^2} -
\alpha\frac{u_t}{u}\le\frac{n\alpha^2k}{2(\alpha-1)}+\frac{n\alpha^2}{2t}.$$

In \cite{D}, Davies improved this estimate to
\begin{equation}\label{Davies}
\frac{|\nabla u|^2}{u^2} - \alpha\frac{u_t}{u}
\le\frac{n\alpha^2k}{4(\alpha-1)}+\frac{n\alpha^2}{2t}.
\end{equation}
Let's denote the right hand side to be $\varphi(t)$. Clearly, when $t$ is big, $\varphi(t)$
converges to $\frac{n\alpha^2k}{4(\alpha-1)}$ which is greater or
equal to $nk$ for any $\alpha>1$. Namely, the optimal estimate for $t$ large one
can get from this estimate is $nk$, which can be obtained by choosing $\alpha=2$. For small time $t$, the dominant term of
$\varphi(t)$ is $\frac{n\alpha^2}{2t}$. By checking examples for
heat kernels on hyperbolic spaces, one finds that when $t$ is small,
the leading term should be $\frac{n}{2t}$. This suggests one should
choose $\alpha$ close to $1$ and when time $t$ is small, the sharp
form is $\frac{n}{2t}$. In (\ref{LY equ2}) and Davies' improved
estimate, if one lets $\alpha \rightarrow 1$, then $\varphi(t)$ will
blow up. This phenomena suggests that there is still room to improve
the estimate.
\\

It is a long time question : can one find a sharp (explicit) form for general manifolds with $Ricci(M)\ge-k$? (see Problem 10.5 in book \cite{Ce}, page 393.) In this paper, we make some progress for this question.  \\

The first main theorem in this paper is the following local gradient
estimate.
\begin{theo}\label{main thm}
Let $(M,g)$ be a complete Riemannian manifold. Let $B_{2R}$ be a
geodesic ball centered at $O\in M$. We assume $Ricci(B_{2R})\ge-k$
with $k\ge0$. If $u$ is a positive solution of the heat equation
\[
(\Delta-\partial_t)u(x,t)=0\quad {\rm on}\quad
B_{2R}\times(0,T],
\]
where $0<T\le \infty$ and let $f=\ln u$, then we get the following Li-Yau
type gradient estimate in $B_R$ \earr\label{main thm equ2}
\sup_{B_R}( |\nabla f|^2 - \alpha f_t-\varphi)(x,t)\le
\frac{nC}{R^2}+\frac{nC\sqrt k}{R}\coth(\sqrt k\cdot
R)+\frac{n^2C}{R^2\tanh(kt)}. \eearr where $C$ depends on $n$, 
$\alpha(t)=1+\frac{\sinh(kt)\cosh(kt)-kt}{\sinh^2(kt)}$ and
$\varphi(t)=\frac{nk}{2}\big[\coth
(kt)+1\big]$.\\

Moveover, letting $R\rightarrow \infty$, if  $Ric(M)\ge -k$ on the
complete manfiold, then \earr\label{main thm equ3}
 |\nabla f|^2 - (1+\frac{\sinh(kt)\cosh(kt)-kt}{\sinh^2(kt)}) f_t\le &\frac{nk}{2}\big[\coth (kt)+1\big] .\\
\eearr
\end{theo}
%\vspace{5mm}

\begin{rema}
When $Ricci\ge0$, letting $k\rightarrow0$, we recovered the celebrated sharp Li-Yau
gradient estimates. Our estimate also explains why in the Li-Yau
gradient estimates for general case (\ref{LY equ1}), one tends to
assume the blow-up parameter $\alpha>1$. The reason is one can view
the parameter $\alpha$ as a function of time $t$, i.e.
$\alpha(t)=1+\frac{\sinh(kt)\cosh(kt)-kt}{\sinh^2(kt)}$, which
indeed is greater than $1$ for all $t>0$ in case of $Ricci(M)\ge-k$
with $k>0$. Since we only assume $k\in \mathbb R$ in the proof, our
estimate in fact also works for Ricci positive case as well. However, in the positive
Ricci case,
$\alpha(t)=1+\frac{\sinh(kt)\cosh(kt)-kt}{\sinh^2(kt)}<1$.
\end{rema}

A linearized version of Theorem \ref{main thm} is the
following.
\begin{theo}\label{linear thm}
Let $(M,g)$ be a complete Riemannian manifold. Let $B_{2R}$ be a
geodesic ball centered at $O\in M$. We assume $Ricci(B_{2R})\ge-k$
with $k\ge0$. If $u$ is a positive solution of the
\[
(\Delta-\partial_t)u(x,t)=0\quad {\rm on}\quad
B_{2R}\times(0,T],
\]
where $0<T\le \infty$ and let $f=\ln u$, then we get the following
Li-Yau type gradient estimate in $B_R$ \earr\label{linear thm equ2}
\sup_{B_R}( |\nabla f|^2 - \alpha f_t-\varphi)(x,t)\le
\frac{C\alpha^2(t)}{R^2}+\frac{C\alpha^2(t)\sqrt k}{R}\coth(\sqrt
k\cdot R)+\frac{C\alpha^4(t)\coth(kt)}{R^2}, \eearr
where $C$ is a constant depending only on $n$, $\alpha=1+\frac{2}{3}kt$ and $\varphi(t)=\frac{n}{2t}+\frac{nk}{2}(1+\frac{1}{3}kt)$.\\

Moveover, letting $R\rightarrow \infty$, if  $Ric(M)\ge -k$ on the
complete manfiold, then \earr\label{linear thm equ3}
 |\nabla f|^2 - (1+\frac{2}{3}kt) f_t\le &\frac{n}{2t}+\frac{nk}{2}(1+\frac{1}{3}kt) .\\
\eearr
\end{theo}

\begin{rema}
The global estimate (\ref{linear thm equ3}) in Theorem \ref{linear
thm} was also obtained in \cite{BQ} by a different method. The local
estimate (\ref{linear thm equ2}) is new. Our proof seems to be
simpler and is more of the local spirit of the classical Li-Yau's
result. Moreover, the method we used can be extended to a matrix
version.
\end{rema}

\begin{rema}
(\ref{linear thm equ3}) is in the same spirit of (\ref{main thm
equ3}) without blow-up parameter $\alpha$. In addition, one can see
from the proof that the first variation vanishes if $M$ is an
Einstein manifold and $u$ satisfies the following gradient Ricci
soliton equation, (a concept first introduced by R. Hamilton in the
study of Ricci flow \cite{Ha93}) \begeq\label{linear thm equality}
\frac{1}{2}R_{ij}-\nabla_i\nabla_j(\ln u)-\frac{1}{2t}g_{ij}\equiv
0.\\
\endeq
%(see (\ref{sLX thm1 equ}) in Theorem \ref{sLX prop1} and remarks in section \ref{sec entropy}.)\\
\end{rema}

In spirit, Theorem \ref{main thm} and Theorem \ref{linear thm} are
very close. The difference is the choice of $\alpha(t)$ and
$\varphi(t)$. Inspecting the following series expansion of functions
$\alpha(t)$ and $\varphi(t)$, one can compare these two theorems.
\[
\begarr
\frac{nk}{2}\big[\coth
(kt)+1\big]&=&\frac{n}{2t}+\frac{1}{2}nk+\frac{nk}{6}(kt)-\frac{nk}{90}(kt)^3+O((kt)^{5})\\
\frac{\sinh(kt)\cosh(kt)-kt}{\sinh^2(kt)}&=&\frac{2}{3}kt-\frac{4}{45}(kt)^3-\frac{4}{315}(kt)^5+O((kt)^7).
\endarr
\]
Indeed, functions in Theorem \ref{linear thm} are the leading terms
of the expansions of functions in Theorem \ref{main thm}. Moreover,
one can show by computations that
$1+\frac{\cosh(kt)\sinh(kt)-kt}{\sinh^2(kt)}\le 1+\frac{2}{3}kt$ and
$\frac{nk}{2}\big[\coth (kt)+1\big]\le
\frac{n}{2t}+\frac{n}{2}k+\frac{n}{6}k^2t$. This implies that
Theorem \ref{main thm} yields sharper estimate than its linearized
version, Theorem \ref{linear thm}.\\

 %Clearly, estimates in (\ref{main thm equ3}) and (\ref{linear thm equ3}) are sharp in the leading term $\frac{n}{2t}$ when $t$ is small.

In \cite{Yau1,Yau2}, Yau established the following gradient
estimate: if $Ricci(M)\ge-k$ with $k\ge0$, then \earr\label{Yau
harnack improved} {|\nabla f|^2} - f_t\le \sqrt{2nk}\sqrt{|\nabla
f|^2+\frac{n}{2t}+2nk}+\frac{n}{2t}, \eearr for all $t>0$. Later,
Bakry-Qian \cite{BQ} improved the inequality to the following
\earr\label{Bakry-Qian harnack improved}
{|\nabla f|^2} - f_t\le \sqrt{nk}\sqrt{|\nabla f|^2+\frac{n}{2t}+\frac{nk}{4}}+\frac{n}{2t}.\\
\eearr

The righthand sides of
(\ref{Yau harnack improved}) and (\ref{Bakry-Qian harnack improved})
blow up as $\frac{n}{2t}+O(\frac{1}{\sqrt t})$ when $t$
is small, while (\ref{main thm equ3}) and (\ref{linear thm equ3})
give sharper estimates which has
blow up order of $\frac{n}{2t}$.\\

In another direction, Hamilton \cite{Ha93matrix} proved

\begeq\label{hamilton} \frac{|\nabla u|^2}{u^2} -
e^{2kt}\frac{u_t}{u} \le e^{4kt}\frac{n}{2t},
\endeq
which is also sharp in the leading term for small $t$. But when
$t\rightarrow \infty$, the righthand sides of (\ref{linear thm
equ3}) and (\ref{hamilton}) will blow up, while the estimate (\ref{main thm
equ3}) in the main theorem stays bounded which implies a better estimate.  In regard of Li-Yau-Davies estimates (\ref{LY
equ2})-(\ref{Davies}) and Hamilton's estimate, one can see that the
new estimate (\ref{main thm equ3}) works for both large and small
time $t$. %See more remarks on the sharpness of the main theorem in Section \ref{sec gradient}.

We can extend (\ref{linear thm equ3}) and (\ref{main thm equ3}) to
the following : under the same hypothesis of Theorem \ref{main thm},
the following holds
\[
\alpha(t)u_t+\varphi(t)u+2Du(V)+u|V|^2\ge0,
\]
for any vector field $V$, where $\alpha(t)$ and $\varphi(t)$ are
defined as in Theorem \ref{main thm} and \ref{linear thm}
respectively. When $k=0$, this form of Li-Yau estimate was first
pointed out in Hamilton's work \cite{Ha93matrix}. Choosing $V\equiv
0$, we get \earr\label{V=0} -\alpha(t)f_t\le \varphi(t), {\rm\quad
for\ all\ }t>0, \eearr where $f=\ln u$. One immediate application of
(\ref{V=0}) for is that
$t^\frac{n}{2}(1+\frac{2}{3}kt)^{-\frac{n}{8}}e^{\frac{n}{4}kt}u$ is
monotonic in $t$. When $k=0$, the monotonicity of $t^\frac{n}{2}u$
is known.\\

%\begin{rema}
%(\ref{sLX thm har}) improves known estimates in the literature. Explicit examples  at the end of section \ref{sec sLX} will show the sharpness of this estimate.
%\end{rema}

%\begin{rema}
%The equality of (\ref{sLX thm equality}) is known as a Ricci soliton equation, which comes from the study of Ricci flow. Or equivalently, $f=\ln u$ satisfies the nonhomogeneous Hamilton-Jacobbi equation $f_t-|\nabla f|^2+\frac{n}{2}(k+\frac{1}{t})=0$.
%\end{rema}
The sharp Li-Yau gradient estimate has tremendous impact in the past twenty years. On one hand, this gradient estimate is a differential Harnack inequality. Namely, it leads to a classical parabolic Harnack inequality which further yields powerful estimates for heat kernels on manifold with nonnegative Ricci curvature. %Heat method may be one of the most important tool in modern mathematics.
There is a vast literature in studying heat kernel even before Li-Yau's work. On the other hand, the idea of Li-Yau leads to Hamilton's Harnack inequalities in the study of Ricci flow which plays a central role in his famous program. We will discuss more along this direction in the second part of this paper.\\

Along the line of Li-Yau, we find applications of the new gradient estimates in deriving Harnack inequalities and new estimates on heat kernels.
%In fact, most of the theorems in the paper of Li-Yau \cite{LY} which consider manifolds with $Ricci(M)\ge-k$ can be improved by applying the new gradient estimates.
For example, we use our gradient estimates to obtain the following Harnack inequatlity.

\begin{theo}\label{Li-Yau Harnack thm1 local}
If $M$ is a complete, noncompact Riemannian manifold with
$Ricci(M)\ge -k$. If $u(x,t):M\times(0,\infty)\rightarrow \mathbb
R^+$ is a positive solution of the heat equation on $M$, then for
$\forall x_1,x_2\in M$, $0<t_1<t_2<\infty$, the following inequality
holds: \begeq\label{Li-Yau Harnack thm1 local equ1} u(x_1,t_1)\le
u(x_2,t_2)A_1(t_1,t_2)\cdot \exp\Bigg[
\frac{dist^2(x_2,x_1)}{4(t_2-t_1)}(1+A_2(t_1,t_2)) \Bigg]
\endeq
where $x_1,x_2\in M$, $0<t_1<t_2<\infty$, $dist(x_1,x_2)$ is the distance between $x_1$ and $x_2$,
$A_1=\Big(\frac{ e^{2kt_2}-2kt_2-1}{ e^{2kt_1}-2kt_1-1}\Big)^\frac{n}{4} $, and $A_2(t_1,t_2)=\frac{t_2\coth (kt_2)-t_1\coth(kt_1)}{t_2-t_1}$.\\
\end{theo}

\begin{rema}
Easy to see $ \di\lim_{k\rightarrow 0}A_1(t_1,t_2)=\Big(
\frac{t_2}{t_1} \Big)^\frac{n}{2}$, $ \di\lim_{k\rightarrow
0}A_2(t_1,t_2)=0. $
%When $k<0$, i.e. $Ricci(M)\ge-k>0$, the Harnack inequality is true for $0\le t< -\frac{3}{k}$, while for $t\in(0,\infty)$ one can still apply the inequality of the $Ricci(M)\ge0$ case.
\end{rema}

Similarly, the linearized gradient estimate also yields a
corresponding Harnack inequality.
\begin{theo}\label{Li-Yau Harnack thm1 linear}

If $M$ is a complete, noncompact Riemannian manifold with
$Ricci(M)\ge -k$. If $u(x,t):M\times(0,\infty)\rightarrow \mathbb
R^+$ is a positive solution of the heat equation on $M$, then for
$\forall x_1,x_2\in M$, $0<t_1<t_2<\infty$, the following inequality
holds: \begeq\label{Li-Yau Harnack thm1 local linear equ1} u(x_1,t_1)\le
u(x_2,t_2)\Big(\frac{t_2}{t_1}
\Big)^\frac{n}{2}\cdot\Big(\frac{1+\frac{2}{3}kt_2}{1+\frac{2}{3}kt_1}
\Big)^{-\frac{n}{8}}\cdot \exp\Big(
\frac{dist^2(x_2,x_1)}{4(t_2-t_1)}\big(1+\frac{1}{3}k(t_2+t_1)\big)+\frac{n}{4}k(t_2-t_1)
\Big)
\endeq
where $x_1,x_2\in M$, $0<t_1<t_2<\infty$, $dist(x_1,x_2)$ is the distance between $x_1$ and $x_2$.\\
\end{theo}

One should compare the above theorems with others when $t$ is large
or small and when $d(x_1,x_2)$ is large or small.
When $k=0$, they reduce to the classical result.\\

As standard, we find a lower bound of the heat kernel as well by using the Harnack inequality.
\begin{theo}\label{lowerbound heat kernel thm1 intro}
Let $M$ be a complete (or compact with convex boundary) Riemannian
manifold possibly with $Ricci(M)\ge -k$. Let $H(x,y,t)$ be the
(Neumann) heat kernel. Then \earr\label{lowerbound heat kernel thm1
equ1 intro} H(x,y,t)\ge& (4\pi
t)^{-\frac{n}{2}}2^{-\frac{n}{4}}\frac{(2kt)^\frac{n}{2}}{(e^{2kt}-2kt-1)^\frac{n}{4}}\cdot
 \exp\Bigg[-
\frac{d^2(x,y)}{4t}\Big(1+\frac{kt\coth (kt)-1}{kt}\Big)
\Bigg]\\
{ and}&\\
 H(x,y,t)\ge& (4\pi
t)^{-\frac{n}{2}}\exp\Big[-\frac{d(x,y)^2}{4t}(1+\frac{1}{3}kt)-\frac{n}{4}kt
\Big], \eearr for all $x,y\in M$ and $t>0$.
\end{theo}
One should compare this theorem with Corollay 2.3 in
\cite{Ha93matrix}.
\begin{rema}
By going through Li-Yau's paper carefully, one can get similar results on the estimates of Greene's function,
lower bounds of Dirichelet or Neumann eigenvalues, betti numbers, etc.
The new contribution will be that explicit dependence of various constants can be established.
\end{rema}
%In particular, in Li-Yau's estimate, the constant $C(\epsilon,n)\rightarrow \infty$, as $\epsilon\rightarrow 0$, while Theorem \ref{Li-Yau heat kernel upper estimates local} gives the constant $C(\delta,n)=(1+\delta)^n\exp\Big(\frac{2}{\delta}\Big)$ which depends only on $\delta$.\\

\vspace{1cm}
%In the field of geometric analysis, maybe the most exciting achievement in the past decades is the discovery of Ricci flow equation by R. Hamilton and the proof of Poincar\'e conjecture and Thurston conjecture by using Ricci flow methods.
In the second part of this paper, we will discuss the relation between Li-Yau type gradient estimate,
Hamilton's gradient estimate, and Perelman type differential Harnack inequality.
Motivated by Li-Yau's fundamental work, Hamilton proved the following gradient estimate.\\
% We modify Hamilton's original statement slightly to fit into our discussion.\\

\noindent{\bf Theorem} \label{LY}{(\bf R. Hamilton \cite{Ha93matrix})}\quad{\em
Let $(M,g)$ be a closed Riemannian manifold with $Ricci(M)\ge-k$. Let $u(x,t)$ be the positive solution to the heat equation. Assume $u\le A$, then
\earr\label{Ha equ1}
t|\nabla u|^2\le(1+2kt)u^2\ln(\frac{A}{u}).\\
\eearr
%where $\beta(t)=(\frac{1}{t}+2k)$.\\
}

Ground breaking progress in the study of Ricci flow and complete proof of Poincar\'e conjecture was made by G. Perelman in 2002-2003. Some important tools which enable Perelman to make the breakthrough were related to Li-Yau \cite{LY}, and Hamilton's earlier work \cite{Ha93matrix,Ha93,Ha93RF}, (see also \cite{CH}). More specifically, Hamilton systematically studied the differential Harnack inequalities in Ricci flow along the line of Li-Yau. Perelman discovered a new sharp differential Harnack inequality for Ricci flow which plays a crucial role in his work. One new feature of Perelman's work is that no curvature assumption is assumed. Moreover, Perelman's differential Harnack is modelled on shrinking Ricci soliton and works for all dimension.
\\

A natural question is whether Perelman's new discovery could shed some lights on the results for linear heat equations. Indeed, one could find highly similarities between backward conjugate heat equation along Ricci flow and heat equation solutions on static Riemmannian manifolds. In \cite{NL1}, one of the main results is the following analogue of Perelman's differential Harnack inequality for heat kernels.\\

%When $Ricci(M)\ge 0$, an analogue of Perelman's differential Harnack inequality for heat kernels was established by L. Ni in \cite{NL1, NL4}.\\

\noindent{\bf Theorem} \label{LY}{(\bf L. Ni \cite{NL1, NL4})}\quad{\em
Let $(M,g)$ be a closed Riemannian manifold with nonnegative Ricci curvature. Let $u(x,t)=H(x,t;y,o)$ be the positive heat kernel. Then
\earr\label{Ni equ1}
t(2\Delta f-|\nabla f|^2)+f-n\le0,
\eearr
where $u=\frac{e^{-f}}{(4\pi t)^\frac{n}{2}}$.\\
}

Easy to see, this type of differential Harnack quantity $t(2\Delta f-|\nabla f|^2)+f-n$ is a hybrid of Li-Yau's estimate on $|\nabla f|^2-\alpha (\Delta f+|\nabla f|^2)$ and Hamilton's estimate on $|\nabla f|^2+(\frac{1}{t}+2k)f$. As we have seen in section \ref{sec Harnack}, from Li-Yau type gradient estimate, one could get a Hamilton-Jacobi inequality which leads to the generalization to a classical parabolic Harnack inequality of Moser. This powerful method was started by Li and Yau. Hamilton extended this method further for heat equations. Moreover, he established similar estimates in the study of Ricci flow. This method now is generally referred as Li-Yau-Hamilton estimate (LYH) (cf. \cite{NL3}).\\

In regard of the nice curvature free feature of Perelman's LYH type
differential Harnack inequality under Ricci flow, and our new
discovery of Li-Yau gradient estimate, one may ask : can one find a
Perelman type of differential Harnack inequality for heat kernels on
any closed Riemannian manifolds? We answer this question
affirmatively. The following is the second main theorem in our
paper.

\begin{theo}\label{thm1.2}
Suppose $M^n$ is a closed manifold. Let $u$ be the positive heat kernel and $k\ge0$ is any constant satisfying $R_{ij}(x)\ge -kg_{ij}$ for all $x\in M^n$, then
\begeq\label{harnack inequ}
\begin{array}{rll}
v:=\big[t\Delta f+t(1+kt)(\Delta f -|\nabla f|^2)+f-n(1+\frac{1}{2}kt)^2\big]u\le 0,
\end{array}
\endeq
for all $t>0$ with $u=\frac{e^{-f}}{(4\pi t)^\frac{n}{2}}$. When $k=0$, this theorem is due to L. Ni.

Moreover, 
\begeq
(\frac{\partial}{\partial t}-\Delta)v=-2t\big|\nabla_i\nabla_j f-(\di\frac{1}{2t}+\frac{k}{2})g_{ij}\big|^2u-2t(R_{ij}+kg_{ij})f_if_ju.
\endeq
\end{theo}

\begin{rema}
The right handside of the evolution equation of $v$ vanishes if the
manifold is Einstein and $f$ satisfies a gradient Ricci soliton
equation. See discussions in section \ref{sec
entropy}.  \\
\end{rema}

\begin{rema}
The proof of the above Perelman type LYH differential Harnack inequality is independent
of the Li-Yau type estimate. In fact, besides the discovery of this
differential Harnack quantity for general manifolds, we have
overcome a technical difficulty in the proof, which in the
$Ricci(M)\ge 0$ case, see \cite{NL1,NL4}, previously one has to appeal to
Perelman's reduced distance function, see Remark \ref{trick}. \\
\end{rema}

%In section \ref{sec LYHP}, we compare three different types of gradient estimates of Li-Yau, Hamilton, and Perelman respectively. We provide a new proof to Hamilton's gradient estimate. As by products, we discover several new differential Harnack inequalities of Perelman type. These LYH inequalities work for all positive solutions of the heat equation instead of just for heat kernels, see Proposition  \ref{hybrid prop} and Remark \ref{hybrid prop rema}. For other examples of Perelman type of LYH differential inequalities, see another recent work of the authors, \cite{LX4}.\\

Point-wise differential Harnack inequalities and monotonicity formulas for entropy functionals are closely related. Usually, a point-wise differential Harnack quantity easily yields a monotonicity formula for the related functional. But reversely, it is more difficult to find the corresponding differential Harnack quantity from a functional monotonicity. In this paper, we will analyze this relation and give various different new entropy monotonicity formulas for heat solution. In fact, this served as one of the motivations of this paper.\\

We introduce the following Li-Yau type entropy formula $\mathcal W_{LY}$ and Perelman type entropy formula $\mathcal W_{P}$, which were formulated from our new point-wise Li-Yau differential Harnack quantity and the new Perelman type differential Harnack quantity respectively.
\earr\label{}
{\mathcal W_{LY}}(u,t)=&-\Dint_{M^n}\sinh^2(kt)\Big[\Delta \ln u + \frac{nk}{2}\big[\coth (kt)+1\big] \Big]ud\mu\\
{\rm or}&\\=&-\Dint_{M^n}t^2\Big[\Delta \ln u + \frac{n}{2t}+\frac{nk}{2}(1+\frac{1}{3}kt) \Big]ud\mu,\\
\\
{\mathcal W_{P}}(f,\tau)=&\Dint_{M^n}\Big(\tau|\nabla f|^2+f-n(1+\textstyle\frac{1}{2}k\tau)^2\Big)\frac{e^{-f}}{(4\pi\tau)^{\frac{n}{2}}}d\mu,
\eearr
where $u$ is a positive solution of the heat equation, $\frac{d\tau}{dt}=1$, and $\frac{e^{-f}}{(4\pi\tau)^{\frac{n}{2}}}=u$ in $\mathcal W_{P}$. We shall prove that these entropy functionals are nonpositive and monotonically nonincreasing for all $t>0$ on manifolds with $Ricci(M)\ge -k$. Moreover, the first variation vanishes if and only if the manifold is Einstein and $\ln u$ satisfies a gradient Ricci soliton equation.\\

We also discuss estimates for Nash entropy on a closed manifold with $Ricci(M)\ge-k$. \\

%Along the line of Perelman's work in Ricci flow, one can define corresponding invariants for the above entropy functionals. These invariants behave like isoperimetric constants of Riemannian manifolds. Hence, the uniform boundedness of them will yield rich geometric information of the underline manifold, such as volume ratial, diameter bounds, etc. We will include parallel discussions to those results in the study of Ricci flow. Our work was motivated by L. Ni's work on manifolds with nonnegative Ricci curvature.\\

\begin{rema}
As in Li-Yau, one could extend all the results in this paper to heat equations with potentials. In particular, one can obtain better Harnack inequality, and lower, upper bounds for the fundamental solution. This was treated by the authors in a separated paper \cite{LX3}.
\end{rema}

This paper is organized as following. In section \ref{sec gradient},
we prove the generalized Li-Yau gradient estimates for manifolds with Ricci curvature bounded from below.
In section \ref{sec Harnack}, we discuss applications of the
gradient estimates and obtain classical Harnack inequalities and
estimates for heat kernels. In section \ref{sec LYHP2}, we prove
the Perelman type Li-Yau-Hamilton Theorem \ref{thm1.2}. In section \ref{sec LYHP Har}, as applications of Theorem \ref{thm1.2},
we derive another parabolic Harnack inequality for
heat kernels. In section \ref{sec entropy}, we discuss
various entropy formulas with monotonicity for heat equations.\\

%In section \ref{sec bdd entropy}, we discuss the relationship between bounded entropy with maximal volume growth conditions.

In this paper, we will use Einstein convention.

\section{\bf Li-Yau type gradient estimates on general manifolds}\label{sec gradient}
We start with the following lemmas.

\begin{lemm}\label{sLY lem1}
Let $M^n$ be a Riemannian manifold, and $u(x,t)$ is a positive
solution of the heat equation. Let $\alpha(t)$ and $\varphi(t)$ be
functions depending on $t$ and $F=|\nabla f|^2-\alpha f_t-\varphi$.
If $f=\ln u$, then
\earr\label{sLY lem1 equ1}
f_t &=& \Delta f+|\nabla f|^2\\
(\Delta -\partial_t)f_t &=& -2\nabla f\nabla f_t\\
(\Delta -\partial_t)|\nabla f|^2&=& 2|f_{ij}|^2-2\nabla |\nabla f|^2\nabla f+2R_{ij}f_if_j \\
(\Delta -\partial_t)F&=& 2|f_{ij}|^2-2\nabla F\nabla
f+2R_{ij}f_if_j+\alpha' f_t +\varphi'
\eearr
\end{lemm}

\begin{proof}
From $u_t=\Delta u$ and $f=\ln u$, we get $f_t=\Delta f + |\nabla
f|^2$.  Hence,
\[
\begarr (\Delta -\partial_t)(\Delta f+|\nabla f|^2)
&=&\Delta(\Delta f+|\nabla f|^2)-\Delta f_t-2\nabla f_t\nabla f\\
&=&-2\nabla f_t\nabla f.
\endarr
\]

The second identity follows from Bochner formula.
\[
\begarr (\Delta -\partial_t)|\nabla f|^2
&=&\Delta|\nabla f|^2-2\nabla f_t\nabla f\\
&=&2|f_{ij}|^2+2\nabla \Delta f\nabla f+2R_{ij}f_if_j-2\nabla f_t\nabla f\\
&=&2|f_{ij}|^2-2\nabla |\nabla f|^2\nabla f+2R_{ij}f_if_j.
\endarr
\]
The last identity in (\ref{sLY lem1 equ1}) follows from the definition of $F$ and the first
three identities.
\end{proof}

\begin{lemm}\label{sLY lem2}
Let $M^n$ be a Riemannian manifold. Suppose $u$, $f$, and $F$ are
defined the same as in Lemma \ref{sLY lem1}. By choosing different
sets of $\alpha(t)$ and $\varphi(t)$, we have the following.
\begin{enumerate}
  \item If $\alpha(t)=1+\frac{2}{3}kt$ and $\varphi(t)=\frac{n}{2t}+\frac{n}{2}k+\frac{n}{6}k^2t$, then
  \earr\label{sLY lem2 identity 1}
  (\Delta -\partial_t)F=&2|f_{ij}+\frac{1}{2t}g_{ij}+\frac{k}{2}g_{ij}|^2-2\nabla F\nabla f+\frac{2}{t}F+2(R_{ij}+kg_{ij})f_if_j\\
  \eearr

  \item If $\alpha(t)=1+\frac{\cosh(kt)\sinh(kt)-kt}{\sinh^2(kt)}$ and $\varphi(t)=\frac{nk}{2}\big[\coth (kt)+1\big]$, then
  \earr\label{sLY lem2 identity 2}
  (\Delta -\partial_t)F=&2|f_{ij}+\frac{\varphi}{n}g_{ij}|^2-2\nabla F\nabla f+2k\coth (kt)F+2(R_{ij}+kg_{ij})f_if_j
  \eearr
\end{enumerate}
\end{lemm}

\begin{rema}
By looking at the series expansion, one can see that the first set
of $\alpha(t)$ and $\varphi(t)$ are the linear approximation of the
second set. Moreover, one can show that
$1+\frac{\cosh(kt)\sinh(kt)-kt}{\sinh^2(kt)}\le 1+\frac{2}{3}kt$ and
$\frac{nk}{2}\big[\coth (kt)+1\big]\le
\frac{n}{2t}+\frac{n}{2}k+\frac{n}{6}k^2t$.
\end{rema}

\begin{proof}
We only prove (\ref{sLY lem2 identity 2}). The proof of (\ref{sLY
lem2 identity 1}) is similar. From lemma \ref{sLY lem1}, we get
\earr\label{sLY lem2 identity 2 proof equ1}
(\Delta -\partial_t)F=& 2|f_{ij}|^2-2\nabla F\nabla f+2R_{ij}f_if_j+\alpha' f_t +\varphi'\\
\\
=&2|f_{ij}+\frac{\varphi}{n}g_{ij}|^2-\frac{4}{n}\varphi\Delta f-\frac{2\varphi^2}{n}-2\nabla F\nabla f\\
\\
&-2k|\nabla f|^2+2(R_{ij}+kg_{ij})f_if_j+\alpha' f_t +\varphi'\\
\\
=&2|f_{ij}+\frac{\varphi}{n}g_{ij}|^2-2\nabla F\nabla f+2(R_{ij}+kg_{ij})f_if_j\\
\\
&-2k|\nabla f|^2-\frac{4}{n}\varphi(f_t-|\nabla f|^2)-\frac{2\varphi^2}{n}+\alpha' f_t +\varphi'\\
\\
=&2|f_{ij}+\frac{\varphi}{n}g_{ij}|^2-2\nabla F\nabla f+2(R_{ij}+kg_{ij})f_if_j\\
\\
&+(\frac{4}{n}\varphi-2k)\Big(|\nabla f|^2-\frac{\frac{4}{n}\varphi-\alpha'}{\frac{4}{n}\varphi-2k}f_t-\varphi\Big)-\frac{2\varphi^2}{n} +\varphi'+(\frac{4}{n}\varphi-2k)\varphi\\
\eearr Direct computations yields that,
$\alpha(t)=1+\frac{\cosh(kt)\sinh(kt)-kt}{\sinh^2(kt)}$ and
$\varphi(t)=\frac{nk}{2}\big[\coth (kt)+1\big]$, satisfy the
following system
\begin{equation}
\left\{
\begin{array}{rll}
\varphi'=&-\frac{2\varphi^2}{n} +2k\varphi\\

\frac{4}{n}\varphi-2k=&2k\coth(kt)\\

\alpha=&\frac{\frac{4}{n}\varphi-\alpha'}{\frac{4}{n}\varphi-2k}
\end{array}
\right.
\end{equation}
Plug into (\ref{sLY lem2 identity 2 proof equ1}), we
get
\[
 (\Delta -\partial_t)F=|f_{ij}+\frac{\varphi}{n}g_{ij}|^2-2\nabla F\nabla f+2k\coth (kt)F+2(R_{ij}+kg_{ij})f_if_j,
\]
which completes the proof.
\end{proof}

We also have

\begin{prop}\label{entropy prop1}
Let $M^n$ be a Riemannian manifold. Suppose $u$ is a positive heat solution, $f=-\ln u$, and $F=|\nabla f|^2+\alpha(t)f_t-\varphi(t)$. By choosing different
sets of $\alpha(t)$ and $\varphi(t)$, we have the following indentities respectively.
\begin{enumerate}
  \item If $\alpha(t)=1+\frac{2}{3}kt$ and $\varphi(t)=\frac{n}{2t}+\frac{n}{2}k+\frac{n}{6}k^2t$, then
  \earr\label{entropy prop1 identity 1}
  (\Delta -\partial_t)t^2Fu=&2t^2|f_{ij}-\frac{1}{2t}g_{ij}-\frac{k}{2}g_{ij}|^2u+2t^2(R_{ij}+kg_{ij})f_if_j\\
  \eearr

  \item If $\alpha(t)=1+\frac{\cosh(kt)\sinh(kt)-kt}{\sinh^2(kt)}$ and $\varphi(t)=\frac{nk}{2}\big[\coth (kt)+1\big]$, then
  \earr\label{entropy prop1 identity 2}
  (\Delta -\partial_t)\sinh^2(kt)Fu=&2\sinh^2(kt)|f_{ij}+\frac{\varphi}{n}g_{ij}|^2u+2\sinh^2(kt)(R_{ij}+kg_{ij})f_if_j.
  \eearr
\end{enumerate}

\end{prop}
\begin{proof}
We will only prove the first identity. The second one is similar. Apply Lemma \ref{sLY lem2} to $u=e^{-f}$, we get 
\[
 (\Delta -\partial_t)F=2|-f_{ij}+\frac{1}{2t}g_{ij}+\frac{k}{2}g_{ij}|^2+2\nabla F\nabla f+\frac{2}{t}F+2(R_{ij}+kg_{ij})f_if_j\]
 Using $u_t=\Delta u$ and direct computations, we can show that 
 \[
   (\Delta -\partial_t)t^2Fu=2t^2|f_{ij}-\frac{1}{2t}g_{ij}-\frac{k}{2}g_{ij}|^2u+2t^2(R_{ij}+kg_{ij})f_if_j.\]

\end{proof}

Consequently, the following estimates on closed
Riemannian manifolds hold.
\begin{theo}\label{closed manifold main}
Let $M^n$ be a compact Riemannian manifold possibly with boundary
and with $Ricci(M)\ge -k$. Suppose $u(x,t)$ is a positive solution
of the heat equation. If $\partial M\neq \emptyset$,  assume that
$\partial M$ is convex, and  $u(x,t)$ satisfies the Neumann boundary
condition
\[
\frac{\partial u}{\partial \nu}=0\quad {\rm on}\quad \partial
M\times(0,\infty),
\]
where $\frac{\partial u}{\partial \nu}$ denotes the outer normal of
$\partial M$. If we let $f=\ln u$, then we have \earr\label{main thm
closed}
 |\nabla f|^2 - (1+\frac{\sinh(kt)\cosh(kt)-kt}{\sinh^2(kt)}) f_t\le &\frac{nk}{2}\big[\coth (kt)+1\big] .\\
\eearr

On the other hand, the following linearized version is also true.
\earr\label{linear thm closed}
 |\nabla f|^2 - (1+\frac{2}{3}kt) f_t\le &\frac{n}{2t}+\frac{n}{2}k+\frac{n}{6}k^2t.\\
\eearr
\end{theo}

\begin{proof}
The proof is by the standard parabolic Maximum Principle (cf. \cite{SY}). We will
skip the details.
\end{proof}

To prove the main theorem Theorem \ref{main thm}, we need the following technical lemma.
\begin{lemm}\label{Lemma
main theorem} For $x>0$, the following is true,
\[
4(\alpha-1)\varphi-n\alpha^2\big(2k\coth(x)-\frac{k}{\cosh(x)\sinh(x)}\big)\le
0,
\]
where $\alpha(x)=1+\frac{\sinh(x)\cosh(x)-x}{\sinh^2(x)}$ and
$\varphi(x)=\frac{nk}{2}\Big(\coth(x)+1\Big)$.
\end{lemm}

\begin{proof}

Equivalently, we need to prove
\[
\begarr
2(\alpha-1)(\frac{\cosh(x)}{\sinh(x)}+1)-\alpha^2\frac{2\cosh^2(x)-1}{\cosh(x)\sinh(x)}\le& 0\\
\\
2(\alpha-1)\cdot (e^{2x}+1)-\alpha^2(e^{2x}+e^{-2x})\le& 0\\
\\
2(\alpha-1)\cdot (e^{4x}+e^{2x})-\alpha^2(e^{4x}+1)\le& 0\\
\\
-e^{4x}(\alpha-1)^2-e^{4x}+2(\alpha-1)e^{2x}-\alpha^2\le&0\\
\\
-e^{4x}(\alpha-1)^2-\big[e^{2x}-(\alpha-1)\big]^2+(\alpha-1)^2-\alpha^2\le&0
\endarr
\]
Since $\alpha=1+\frac{\sinh(x)\cosh(x)-x}{sinh^2(x)}$, we get
$0\le\alpha-1\le \alpha$. Hence the last inequality is true, which
finishes the proof of this lemma.
\end{proof}

We now prove the main theorem,  Theorem \ref{main thm}.\\

\begin{proof}{\bf(Proof of Theorem \ref{main thm})}
Let's denote $F=|\nabla f|^2-\alpha(t)f_t-\varphi(t)$ and
$G=\beta(t)F$, where $\alpha(t)$ and $\varphi(t)$ are defined as in
the main theorem and $\beta(t)$ is a positive function of $t$ to be determined.
Applying $(1)$ in Lemma \ref{sLY lem2} to $G=\beta(t)F$, we obtain
the following \earr\label{equation for G}
(\Delta -\partial_t)G=&2\beta|f_{ij}+\frac{\varphi}{n}g_{ij}|^2-2\nabla G\nabla f+G\big(2k\coth(kt)-\frac{\beta'}{\beta}\big)+2\beta (R_{ij}-k)f_if_j\\
\ge&2\beta|f_{ij}+\frac{\varphi}{n}g_{ij}|^2-2\nabla G\nabla
f+G\big(2k\coth(kt)-\frac{\beta'}{\beta}\big).

\eearr \\

Recall that one can construct a cut-off function $\phi$ as in the
proof of Theorem 4.2 in \cite{SY}, which satisfies {\em supp}
$\phi\subset B_{2R}$ and $\phi\big|B_R\equiv 1$. Moreover,
\earr\label{local version main proof equ1}
\frac{|\nabla \phi|^2}{\phi}&\le& \frac{C}{R^2}\\
\\
\Delta \phi&\ge&-\frac{C}{R^2}\Big(1+\sqrt k\cdot R\coth(\sqrt
k\cdot R)\Big), \eearr 
where $C$ depends only on $n$. We want to apply maximum principle to $\phi
G$ on $B_{2R}\times[0,T]$. If $\phi G$ attains its maximum at
$(x_0,t_0)\in B_{2R}\times [0,T]$, then $(\phi G)(x_0,t_0)>0$
without loss of generality. So $x_0\in B_{2R}$, $t_0>0$, and by
maximum principle, at $(x_0,t_0)$, \earr\label{local version main
proof equ2}
0=\nabla(\phi G)&=&G\nabla \phi +\phi \nabla G\\
\\
\Delta(\phi G)&\le&0\\
\\
\partial_t(\phi G)&=&\phi G_t\ge0.
\eearr

Notice that for an $n\times n$ matrix $A$, we have
$|A|^2\geq\frac{1}{n}\Big({\bf tr} A\Big)^2$, and
\begin{eqnarray}
\qquad {\bf tr}\big(f_{ij}+\frac{\varphi}{n}g_{ij}\big)=\Delta
f+\varphi=-\frac{1}{\alpha}\Big[\frac{G}{\beta}+(\alpha -1)(|\nabla
f|^2-\varphi)\Big] \label{trace}
\end{eqnarray}

In the sequel, all computations will be at the maximal point
$(x_0,t_0)$ and we will frequently use property (\ref{local version
main proof equ2}) whenever necessary. Applying (\ref{equation for
G}), we have \earr\nonumber
0\ge& (\Delta-\partial_t) (\phi G)= G\Delta\phi -2G\frac{|\nabla \phi|^2}{\phi}+\phi (\Delta -\partial_t) G\\
\\
\ge &G\Big(\Delta\phi -2\frac{|\nabla
\phi|^2}{\phi}\Big)-2G\nabla\phi \nabla f
+2\phi\beta\big|f_{ij}+\frac{\varphi}{n}g_{ij}\big|^2+\phi G\cdot
\big(2k\coth(kt)-\frac{\beta'}{\beta}) \eearr

Multiplying by $\phi$, and applying (\ref{trace}), we have
\earr\nonumber
0\ge &(\phi G)\Big(\Delta\phi -2\frac{|\nabla \phi|^2}{\phi}+ 2k\coth(kt)\phi-\frac{\beta'}{\beta}\Big)-2(\phi G)|\nabla \phi||\nabla f|\\
&+\frac{2\phi^2\beta}{n\alpha^2} \Big[\frac{G}{\beta}+(\alpha -1)(|\nabla f|^2-\varphi)\Big]^2\\
\\
=&(\phi G)\Big(\Delta\phi -2\frac{|\nabla \phi|^2}{\phi}+ 2k\coth(kt)\phi-\frac{\beta'}{\beta}\Big)-2(\phi G)|\nabla \phi||\nabla f|\\
\\
&+\frac{2\phi^2}{n\alpha^2\beta}G^2+ \frac{2\phi^2(\alpha -1)^2\beta}{n\alpha^2}(|\nabla f|^2-\varphi)^2+(\phi G)\cdot\frac{4\phi(\alpha-1)}{n\alpha^2}(|\nabla f|^2-\varphi)\\
\\
=&(\phi G)\Big(\Delta\phi -2\frac{|\nabla \phi|^2}{\phi}+
2k\coth(kt)\phi-\frac{\beta'}{\beta}-\frac{4\phi(\alpha-1)}{n\alpha^2}\varphi\Big)+
\frac{2\phi^2}{n\alpha^2\beta}G^2\\
\\
&+ \frac{2\phi^2(\alpha -1)^2\beta}{n\alpha^2}(|\nabla f|^2-\varphi)^2+(\phi G)\cdot(\frac{4\phi(\alpha-1)}{n\alpha^2}|\nabla
f|^2-2|\nabla \phi||\nabla f|) \eearr

Applying inequality $ax^2-bx\ge -\frac{b^2}{4a}$, $(a>0)$, to the
last term, and also drop the second last term which is nonnegative,
we get \earr\nonumber 0\ge &(\phi G)\Big(\Delta\phi -2\frac{|\nabla
\phi|^2}{\phi}+
2k\coth(kt)\phi-\frac{\beta'}{\beta}-\frac{4\phi(\alpha-1)}{n\alpha^2}\varphi-\frac{n\alpha^2|\nabla
\phi|^2}{4\phi(\alpha-1)}\Big)+\frac{2}{n\alpha^2\beta}(\phi G)^2.
\eearr

Since $(\phi G)(x_0,t_0)>0$, we get
\earr
(\phi G)(x_0,t_0)\leq&
\frac{\beta}{2}\Big[4(\alpha-1)\varphi-n\alpha^2\big(2k\coth(kt)-\frac{\beta'}{\beta}\big)\Big]\phi\\
&+\frac{n\alpha^2\beta}{2}\Big(-\Delta\phi
+2\frac{|\nabla \phi|^2}{\phi}+\frac{n\alpha^2|\nabla
\phi|^2}{4\phi(\alpha-1)} \Big),
\eearr
where the righthand side is evaluated at $(x_0,t_0)$ which depends
on the function.

Hence on $B_R\times [0,T]$, applying estimates (\ref{local version
main proof equ1}) on $\phi$, we have \earr\label{local version main
proof equ3}
G(x,t)\le& (\phi G)(x_0,t_0)\\
\\
\le&\frac{\beta}{2}\Big[4(\alpha-1)\varphi-n\alpha^2\big(2k\coth(kt)-\frac{\beta'}{\beta}\big)\Big]\phi
\\
&+\frac{n\alpha^2\beta}{2}\Big(\frac{3C}{R^2}+\frac{C\sqrt
k}{R}\coth(\sqrt k\cdot
R)+\frac{n\alpha^2}{4(\alpha-1)}\frac{C}{R^2} \Big). \eearr

Next, we choose $\beta(t)=\tanh(kt)$. Hence,
$\frac{\beta'}{\beta}=\frac{k}{\sinh(kt)\cosh(kt)}$. Denote $x=kt$
and applying Lemma \ref{Lemma main theorem}, we have
\earr\label{local version main proof equ4}
4(\alpha-1)\varphi-n\alpha^2\big(2k\coth(kt)-\frac{\beta'}{\beta}\big)\le
0. \eearr

On the other hand, by definitions,
$\beta(t)=\tanh(kt)$, $\alpha(t)-1=\frac{\sinh(kt)\cosh(kt)-kt}{\sinh^2(kt)}$, as $t\rightarrow 0$, we have $\frac{\beta\alpha^4}{\alpha-1}\rightarrow 2$; as $t\rightarrow
\infty$,  we have $\frac{\beta\alpha^4}{\alpha-1}\rightarrow 1$. This
implies
\begin{equation}\label{local version main proof equ4.1}\frac{\beta\alpha^4}{\alpha-1}\le C,\end{equation}
where $C$ is a universal constant.

Recall that all the computations are at $(x_0,t_0)$ and $(x_0,t_0)$
is the maximum point, $t_0\le T$ and $\beta(t)=\tanh(kt)$ is
non-decreasing. Plug (\ref{local version main proof equ4}) and
(\ref{local version main proof equ4.1}) into (\ref{local version
main proof equ3}), we get
\[
\begarr (\phi G)(x,T)\le &(\phi
G)(x_0,t_0)\le\frac{n\alpha^2(t_0)\beta(t_0)}{2}\Big(\frac{3C}{R^2}+\frac{C\sqrt
k}{R}\coth(\sqrt k\cdot R)\Big)+\frac{n^2C}{R^2}
\\
\le&\beta(T)\Big(\frac{nC}{R^2}+\frac{nC\sqrt k}{R}\coth(\sqrt
k\cdot R)\Big)+\frac{n^2C}{R^2},
\endarr
\]
where in the last inequality, we have used the fact that $\alpha(t)$ is uniformly bounded over $(0,\infty)$. 
But $\phi\equiv1$ on $B_R$, hence, from $G=\beta F$, we have
\[
\begarr \sup_{B_R}F(x,T)\le&\frac{nC}{R^2}+\frac{nC\sqrt
k}{R}\coth(\sqrt k\cdot R)+\frac{n^2C}{R^2\tanh(kT)},
\endarr
\]
since $T$ is arbitrary, the theorem has been proved.

\end{proof}

Similarly, choosing a different set of $\alpha(t)$ and $\varphi(t)$,
one can prove the linearized local version, Theorem \ref{linear thm}.

\begin{proof}{\bf(Proof of Theorem \ref{linear thm})}
The proof is similar to the proof of Theorem \ref{main thm}. The
difference is the choices of $\varphi(t)$ and $\alpha(t)$. Denote
$F=|\nabla f|^2-\alpha(t)f_t-\varphi(t)$ and $G=\beta F$, where
$\alpha(t)=1+\frac{2}{3}kt$ and
$\varphi(t)=\frac{n}{2t}+\frac{nk}{2}(1+\frac{1}{3}kt)$ are defined
as in Theorem \ref{linear thm}. Applying $(1)$ of Lemma \ref{sLY
lem2} to $G=tF$, we obtain the following \earr\label{equation for G
linear}
(\Delta -\partial_t)G=&2\beta|f_{ij}+(\frac{1}{2t}+\frac{k}{2})g_{ij}|^2-2\nabla G\nabla f+\Big(\frac{2}{t}-\frac{\beta'}{\beta}\Big)G+2\beta (R_{ij}-k)f_if_j\\
\ge&2\beta|f_{ij}+(\frac{1}{2t}+\frac{k}{2})g_{ij}|^2-2\nabla
G\nabla f+\Big(\frac{2}{t}-\frac{\beta'}{\beta}\Big)G.

\eearr \\

Construct the cut-off function $\phi$ as before, which satisfies
{\em supp} $\phi\subset B_{2R}$ and $\phi\big|B_R\equiv 1$.
Moreover, \earr\label{local version linear proof equ1}
\frac{|\nabla \phi|^2}{\phi}&\le& \frac{C}{R^2}\\
\\
\Delta \phi&\ge&-\frac{C}{R^2}\Big(1+\sqrt k\cdot R\coth(\sqrt
k\cdot R)\Big). \eearr

Apply maximum principle to $\phi G$ on $B_{2R}\times[0,T]$. If $\phi
G$ attains its maximum at $(x_0,t_0)\in B_{2R}\times [0,T]$, then
$(\phi G)(x_0,t_0)>0$ without loss of generality. So $x_0\in
B_{2R}$, $t_0>0$, and by maximum principle, at $(x_0,t_0)$
\earr\label{local version linear proof equ2}
0=\nabla(\phi G)&=&G\nabla \phi +\phi \nabla G\\
\\
\Delta(\phi G)&\le&0\\
\\
\partial_t(\phi G)&=&\phi G_t\ge0.
\eearr

Using the matrix inequality $|A|^2\geq\frac{1}{n}\Big({\bf tr}
A\Big)^2$, we have
\begin{equation} \label{trace linear}
\qquad {\bf
tr}\big(f_{ij}+(\frac{1}{2t}+\frac{k}{2})g_{ij}\big)=\Delta
f+\frac{n}{2}(\frac{1}{t}+k)=-\frac{1}{\alpha}\Big[\frac{G}{\beta}+(\alpha
-1)|\nabla f|^2-(\frac{nk}{3}+\frac{1}{6}nk^2t)\Big]
\end{equation}

In the sequel, all computations will be at the maximal point
$(x_0,t_0)$. Applying (\ref{equation for G linear}), we have
\earr\nonumber
0\ge& (\Delta-\partial_t) (\phi G)= G\Delta\phi -2G\frac{|\nabla \phi|^2}{\phi}+\phi (\Delta -\partial_t) G\\
\\
\ge &G\Big(\Delta\phi -2\frac{|\nabla
\phi|^2}{\phi}\Big)-2G\nabla\phi \nabla f +2\phi
\beta\big|f_{ij}+(\frac{1}{2t}+\frac{k}{2})g_{ij}\big|^2+\Big(\frac{2}{t}-\frac{\beta'}{\beta}\Big)\phi
G \eearr

Multiplying by $\phi$, and applying (\ref{trace linear}), we have
\earr\nonumber
0\ge &(\phi G)\Big(\Delta\phi -2\frac{|\nabla \phi|^2}{\phi}+ \big(\frac{2}{t}-\frac{\beta'}{\beta}\big)\phi\Big)-2(\phi G)|
\nabla \phi||\nabla f|\\
&+\frac{2\phi^2\beta}{n\alpha^2} \Big[\frac{G}{\beta}+(\alpha -1)|\nabla f|^2-(\frac{nk}{3}+\frac{1}{6}nk^2t)\Big]^2\\
\\
=&(\phi G)\Big(\Delta\phi -2\frac{|\nabla \phi|^2}{\phi}+ \big(\frac{2}{t}-\frac{\beta'}{\beta}\big)
\phi\Big)-2(\phi G)|\nabla \phi||\nabla f|+\frac{2\phi^2}{n\alpha^2\beta}G^2\\
\\
&+ \frac{2\phi^2\beta}{n\alpha^2}\Big[(\alpha -1)|\nabla f|^2-(\frac{nk}{3}+\frac{1}{6}nk^2t) \Big]^2+(\phi G)\cdot\frac{4\phi}{n\alpha^2}\Big[(\alpha -1)|\nabla f|^2-(\frac{nk}{3}+\frac{1}{6}nk^2t) \Big]\\
\\
=&(\phi G)\Big(\Delta\phi -2\frac{|\nabla \phi|^2}{\phi}+
\big(\frac{2}{t}-\frac{\beta'}{\beta}\big)\phi-\frac{4}{n\alpha^2}(\frac{nk}{3}+\frac{1}{6}nk^2t)\phi\Big)+
\frac{2\phi^2}{n\alpha^2\beta}G^2\\
\\&+ \frac{2\phi^2\beta}{n\alpha^2}\Big[(\alpha -1)|\nabla f|^2-(\frac{nk}{3}+\frac{1}{6}nk^2t) \Big]^2
+(\phi G)\cdot(\frac{4\phi(\alpha-1)}{n\alpha^2}|\nabla
f|^2-2|\nabla \phi||\nabla f|) \eearr

Applying inequality $ax^2-bx\ge -\frac{b^2}{4a}$, $(a>0)$, to the
last term, and also drop the second last term which is nonnegative,
we get \earr\nonumber 0\ge &(\phi G)\Big(\Delta\phi -2\frac{|\nabla
\phi|^2}{\phi}+
\big(\frac{2}{t}-\frac{\beta'}{\beta}\big)\phi-\frac{2}{\alpha^2}(\frac{2k}{3}+\frac{1}{3}k^2t)\phi
-\frac{n\alpha^2|\nabla
\phi|^2}{4\phi(\alpha-1)}\Big)+\frac{2}{n\alpha^2\beta}(\phi G)^2.
\eearr

Since $(\phi G)(x_0,t_0)>0$, we get
\begin{eqnarray}
(\phi G)(x_0,t_0)\leq
\frac{n\alpha^2\beta}{2}\Big[\big(\frac{\beta'}{\beta}-\frac{2}{t}\big)+\frac{2}{\alpha^2}(\frac{2k}{3}+\frac{1}{3}k^2t)
\Big]\phi+ \frac{n\alpha^2\beta}{2}\Big(-\Delta\phi +2\frac{|\nabla
\phi|^2}{\phi}+\frac{n\alpha^2|\nabla \phi|^2}{4\phi(\alpha-1)}
\Big)\nonumber,
\end{eqnarray}
where the righthand side is evaluated at $(x_0,t_0)$ which depends
on the function.

Hence on $B_R\times [0,T]$, applying estimates (\ref{local version
main proof equ1}) on $\phi$, we have \earr\label{local version
linear proof equ3}
G(x,t)\le& (\phi G)(x_0,t_0)\\
\\
\le&\frac{n\beta}{2t}\Big[t\alpha^2\big(\frac{\beta'}{\beta}-\frac{2}{t}\big)+2(\frac{2kt}{3}+\frac{1}{3}k^2t^2)\Big]\phi
+\frac{n\alpha^2\beta}{2}\Big(\frac{3C}{R^2}+\frac{C\sqrt
k}{R}\coth(\sqrt k\cdot
R)+\frac{n\alpha^2}{4(\alpha-1)}\frac{C}{R^2} \Big). \eearr

Next, we choose $\beta(t)=\tanh(kt)$. Then,
$\frac{\beta'}{\beta}=\frac{k}{\sinh(kt)\cosh(kt)}$. Denote $x=kt$
and recall $\alpha(t)=1+\frac{2}{3}kt$. It is not hard to show that
for $x>0$, (see the comments after the proof of this theorem), \earr\label{local version linear proof equ4}

(\frac{x}{\sinh(x)\cosh(x)}-2)(1+\frac{4}{3}x+\frac{4}{9}x^2)+(\frac{4}{3}x+\frac{6}{9}x^2))\le
0, \eearr this yields that
\[
t\alpha^2\big(\frac{\beta'}{\beta}-\frac{2}{t}\big)+2(\frac{2kt}{3}+\frac{1}{3}k^2t^2)\le0.
\]

On the other hand, by definitions, as $t\rightarrow 0$,
$\beta(t)=O(t)$ and $\alpha(t)-1=\frac{2}{3}kt$. This implies
\begin{equation}\label{local version linear proof equ4.1}\frac{\beta}{\alpha-1}\le C,\end{equation}
where $C$ is a constant.

Recall that all the computations are at $(x_0,t_0)$ and $(x_0,t_0)$
is the maximum point, $t_0\le T$, $\alpha(t)=1+\frac{2}{3}kt$ and
$\beta(t)=\tanh(kt)$ are non-decreasing. Plug (\ref{local version
linear proof equ4}) and (\ref{local version linear proof equ4.1})
into (\ref{local version linear proof equ3}), we get
\[\begarr
(\phi G)(x,T)\le&(\phi G)(x_0,t_0)\le
\frac{n\alpha^2(t_0)\beta(t_0)}{2}\Big(\frac{3C}{R^2}+\frac{C\sqrt
k}{R}\coth(\sqrt k\cdot
R)\Big)+\frac{Cn^2\alpha^4(t_0)}{R^2} \\
\le&\frac{n\alpha^2(T)\beta(T)}{2}\Big(\frac{3C}{R^2}+\frac{C\sqrt
k}{R}\coth(\sqrt k\cdot R)\Big)+\frac{Cn^2\alpha^4(T)}{R^2}.
\endarr\]  Since $\phi\equiv1$ on $B_R$, we obtain the
following gradient estimate in $B_R$
\[
\begarr
\sup_{B_R}F(x,T)\le&\frac{C\alpha^2(T)}{R^2}+\frac{C\alpha^2(T)\sqrt
k}{R}\coth(\sqrt k\cdot R)+\frac{C\alpha^4(T)}{R^2\tanh(kT)}.
\endarr
\]
Since $T$ is arbitrary, we proved the first part of the theorem.

If the manifold is complete, for any fixed $T>0$, letting
$R\rightarrow \infty$, we get
\[
\begarr F(x,T)\le&0.
\endarr
\]
Since $T$ is arbitrary, equivalently, we obtain the global estimate
(\ref{linear thm equ3}).
\end{proof}

One way of proving (\ref{local version linear proof equ4})
is to prove an equivalent inequality as follows, \earr\label{local
version linear proof equ4.2}

I(x):=(e^{2x}-e^{-2x})(1+\frac{2}{3}x+\frac{1}{9}x^2)-2x(1+\frac{4}{3}x+\frac{4}{9}x^2)\ge
0, \eearr where function $I(x)$ is a real analytic function and all
the coefficients of its Taylor expansion are positive.
\section{\bf Harnack inequality and heat kernel estimates} \label{sec Harnack}
%\subsection{\bf New Li-Yau Harnack inequalities}
Along the line of Li-Yau, as an application of the gradient estimates in section \ref{sec gradient}, we can establish Harnack inequalities for positive solutions of the heat equation and deduce lower and upper bounds for the heat kernel. \\

We prove the Harnack inequality for noncompact manifold first.

\begin{proof}{\bf(Proof of Theorem \ref{Li-Yau Harnack thm1 local} and Theorem \ref{Li-Yau Harnack thm1
linear})} Let $f=\ln u$, then combining estimate (\ref{main thm
equ3}) and the heat equation, we have the following Hamilton-Jacobi
inequality,
\[
f_t\ge -\frac{1}{\alpha}(\varphi(t)-|\nabla f|^2).
\]

%%%%%%%%%%%%%%%%%%%%%%%%%%%%%%%%%%
%{\em Proof follows from Chow's book. Be careful of the wording.}\\
%%%%%%%%%%%%%%%%%%%%%%%%%%%%%%%%%%

\iffalse

Take $\gamma $ be a shortest geodesic joining $x_1$ and $x_2$ with constant speed so that $|\frac{d\gamma}{dt}|=\frac{d(x_1,x_2)}{t_2-t_1}$.
\[
\begarr
f(x_2,t_2)-f(x_1,t_1)&=&\Dint^{t_2}_{t_1}\frac{d}{dt}f(\gamma(t))dt\\
\\
&=&\Dint^{t_2}_{t_1}\Big(\<\frac{d\gamma}{dt},\nabla f\>+f_t\Big)dt\\
\\
&\ge&\Dint^{t_2}_{t_1}\Big(-\bigg|\frac{d\gamma}{dt}\bigg|\cdot\big|\nabla f\big|-\frac{1}{\alpha}\big(\varphi(t)-|\nabla f|^2\big)\Big)dt\\
\\
&\ge&-\Dint^{t_2}_{t_1}\Big(\frac{\alpha}{4}
\bigg|\frac{d\gamma}{dt}\bigg|^2+\frac{1}{\alpha}\varphi(t)\Big)dt\\
\\
&=& -\Big(\frac{t+\frac{1}{3}kt^2}{4}\frac{d^2(x_1,x_2)}{(t_2-t_1)^2}+\frac{n}{2}\ln t +\frac{3n}{4}kt+\frac{n}{8}\ln(1+\frac{2}{3}kt)\Big)\Bigg|^{t_2}_{t_1}.
\endarr
\]
Taking the exponential of both sides and flip the quotient, we complete the proof.

\fi

%%%%%%%%%%%%%%%%%%%%%%%%%%%%%%%%%%
%{\em Original proof of Li-Yau.\\}
%%%%%%%%%%%%%%%%%%%%%%%%%%%%%%%%%%

Let $\gamma $ be a shortest geodesic joining $x_1$ and $x_2$, $\gamma:[0,1]\rightarrow M$, $\gamma(0)=x_2$ $\gamma(1)=x_1$. Define a curve $\eta$ in $M\times (0,\infty)$, $\eta  :[0,1]\rightarrow M\times (0,\infty)$ by $\eta(s)=\big( \gamma(s), (1-s)t_2+st_1 \big)$. We have $\eta(0)=(x_2,t_2)$, $\eta(1)=(x_1,t_1)$.  If $\rho=d(x_1,x_2)$, then $|\dot \gamma|=\rho$. We have
\[
\begarr
f(x_1,t_1)-f(x_2,t_2)&=&\Dint_{0}^{1}\frac{d}{ds}f(\eta(s))ds\\
\\
&=&\Dint_{0}^{1}\Big(\<\dot \gamma,\nabla f\>-(t_2-t_1)f_t\Big)ds\\
\\
&\le&\Dint_{0}^{1}\Big(\rho|\nabla f|+\frac{t_2-t_1}{\alpha}(\varphi(t)-|\nabla f|^2)\Big)ds,
\endarr
\]
where $t=(1-s)t_2+st_1$.

The integrand is a quadratic polynomial in $|\nabla f|$, whose maximum value is
\[
\frac{\alpha \rho^2}{4(t_2-t_1)}+\frac{t_2-t_1}{\alpha}\varphi(t).
\]

Therefore, we obtain
\[
\begarr
f(x_1,t_1)-f(x_2,t_2)&\le&\Dint^{1}_{0}\Big(\frac{\rho^2}{4(t_2-t_1)}
\alpha(t)+\frac{t_2-t_1}{\alpha}\varphi(t)\Big)ds\\
\\
&=&\frac{\rho^2}{4(t_2-t_1)^2}\Dint^{t_2}_{t_1}
\alpha(t)dt+\Dint^{t_2}_{t_1}\frac{\varphi(t)}{\alpha}dt\\
\\
&=& \Big[\frac{\rho^2}{4(t_2-t_1)^2}\big(t+\frac{kt\coth(kt)-1}{k}\big)+\frac{n}{4}\ln\frac{ \sinh(2kt)-2kt+\cosh(2kt)-1}{2k}\Big]\Bigg|^{t_2}_{t_1},
\endarr
\]
where in the second identity we have used $t({s=0})=t_2$,
$t({s=1})=t_1$, $dt=-(t_2-t_1)ds$, and we have chosen
$\alpha(t)=1+\frac{\sinh(kt)\cosh(kt)-kt}{\sinh^2(kt)}$ and
$\varphi(t)=\frac{nk}{2}\big[\coth (kt)+1\big]$. Taking the
exponential of both sides and flip the quotient, we finish the proof
of Theorem \ref{Li-Yau Harnack thm1 local}. Similarly, choosing
$\alpha=1+\frac{2}{3}kt$ and
$\varphi(t)=\frac{n}{2t}+\frac{nk}{2}(1+\frac{1}{3}kt)$, we prove
Theorem \ref{Li-Yau Harnack thm1 linear}.

%%%%%%%%%%%%%%%%%%%%%%%%%
%%%%%%%%%%%%%%%%%%%%%%%%%
\end{proof}

Theorem \ref{Li-Yau Harnack thm1 local} extends Li-Yau Harnack estimate in several ways. When $k=0$, we recover Li-Yau's theorem of the case $Ricci(M)\ge0$. When $Ricci(M)\ge-k$ and $k>0$, this theorem finds the explicit form for Li-Yau's original theorem without parameters and it improves previous estimates.  When $Ricci(M)>0$, this formula is new. In the last case, the Harnack is true only for short time.\\

For a positive solution on compact manifolds, the Harnack inequality given by Theorem \ref{Li-Yau Harnack thm1 local} also holds.
\\

\begin{theo}\label{Li-Yau Harnack thm1 compact}
If $M$ is a compact Riemannian manifold possibly with boundary and
with $Ricci(M)\ge -k$. If $\partial M\neq \emptyset$, we assume
$\partial M$ is convex. $u(x,t)$ is a positive solution of the heat
equation on $M$, and $\frac{\partial u}{\partial \nu}=0$ if
$\partial M\neq \emptyset$. \begeq u(x_1,t_1)\le
u(x_2,t_2)A_1(t_1,t_2)\cdot \exp\Bigg[
\frac{dist^2(x_2,x_1)}{4(t_2-t_1)}(1+A_2(t_1,t_2)) \Bigg]
\endeq
where $x_1,x_2\in M$, $0<t_1<t_2<\infty$, $dist(x_1,x_2)$ is the
distance between $x_1$ and $x_2$,
$A_1=\Big(\frac{ e^{2kt_2}-2kt_2-1}{ e^{2kt_1}-2kt_1-1}\Big)^\frac{n}{4} $, and
$A_2(t_1,t_2)=\frac{t_2\coth (kt_2)-t_1\coth(kt_1)}{t_2-t_1}$.\\

\end{theo}
\begin{proof}
The proof is similar to the proof of Theorem \ref{Li-Yau Harnack thm1 local}. We skip the details.\\
\end{proof}

Applying the linearized gradient estimate Theorem \ref{linear thm},
we obtain another Harnack inequality on compact manifolds as follows. The proof is
similar which we will also skip.

\begin{theo}\label{Li-Yau Harnack thm1 compact linear}
If $M$ is a compact Riemannian manifold possibly with boundary and with $Ricci(M)\ge -k$. If $\partial M\neq \emptyset$, we assume $\partial M$ is convex. $u(x,t)$ is a positive solution of the heat equation on $M$, and $\frac{\partial u}{\partial \nu}=0$ if $\partial M\neq \emptyset$. Then
\begeq\label{Li-Yau Harnack thm1 equ1}
u(x_1,t_1)\le u(x_2,t_2)\Big(\frac{t_2}{t_1}  \Big)^\frac{n}{2}\cdot
\exp\Big(  \frac{dist^2(x_2,x_1)}{4(t_2-t_1)}\big(1+\frac{1}{3}k(t_2+t_1)\big)+\frac{n}{4}k(t_2-t_1) \Big)
\endeq
where $x_1,x_2\in M$, $0<t_1<t_2<\infty$, $dist(x_1,x_2)$ is the distance between $x_1$ and $x_2$.
\end{theo}

It is well-known that Harnack inequality leads to lower bounds on
the heat kernel. Applying the Harnack estimates, we get\\

\noindent{\em {\bf Theorem \ref{lowerbound heat kernel thm1 intro}.}
Let $M$ be a complete (or compact with convex boundary) Riemannian
manifold possibly with $Ricci(M)\ge -k$. Let $H(x,y,t)$ be the
(Neumann) heat kernel. Then \earr\label{lowerbound heat kernel thm1
equ1 intro} H(x,y,t)\ge& (4\pi
t)^{-\frac{n}{2}}2^{-\frac{n}{4}}\frac{(2kt)^\frac{n}{2}}{(e^{2kt}-2kt-1)^\frac{n}{4}}\cdot
 \exp\Bigg[-
\frac{d^2(x_2,x_1)}{4t}\Big(1+\frac{kt\coth (kt)-1}{kt}\Big)
\Bigg]\\
{ and}&\\
 H(x,y,t)\ge& (4\pi
t)^{-\frac{n}{2}}\exp\Big[-\frac{d(x,y)^2}{4t}(1+\frac{1}{3}kt)-\frac{n}{4}kt
\Big], \eearr for all $x,y\in M$ and $t>0$. }

\begin{proof}
Since the proofs are similar, we will only prove the first
inequality for complete noncompact manifolds. Applying the Harnack inequality (\ref{Li-Yau Harnack
thm1 local equ1}) to the function
\[
u(y,t)=H(x,y,t),
\]
we obtain
\[
\begarr H(x,x,s)\le& H(x,y,t+s)\Big(\frac{ e^{2k(t+s)}-2k(t+s)-1}{
e^{2ks}-2ks-1}\Big)^\frac{n}{4}\\
&\cdot
 \exp\Bigg[
\frac{d^2(x_2,x_1)}{4t}\Big(1+\frac{(t+s)\coth
(k(t+s))-s\coth(ks)}{t}\Big) \Bigg],
\endarr
\]
for all $s>0$ and $t>0$. By local calculations one gets
\[
\di\lim_{s\rightarrow0}(4\pi s)^\frac{n}{2}H(x,x,s)=1,
\]
so, multiplying by $\lim_{s\rightarrow0}(4\pi s)^\frac{n}{2}$, we
get
\[
\begarr (4\pi s)^\frac{n}{2}H(x,x,s)\le&\di (4\pi t)^\frac{n}{2}
H(x,y,t+s)\frac{s^\frac{n}{2}}{(e^{2ks}-2ks-1)^\frac{n}{4}}\frac{(e^{2k(t+s)}-2k(t+s)-1)^\frac{n}{4}}{t^\frac{n}{2}}\\
&\cdot
 \exp\Bigg[
\frac{d^2(x_2,x_1)}{4t}\Big(1+\frac{(t+s)\coth
(k(t+s))-s\coth(ks)}{t}\Big) \Bigg]. \endarr
\]

Let $s\rightarrow 0$, we obtain
\[
1\le\di (4\pi t)^\frac{n}{2}
H(x,y,t)2^{\frac{n}{4}}\frac{(e^{2kt}-2kt-1)^\frac{n}{4}}{(2kt)^\frac{n}{2}}\cdot
 \exp\Bigg[
\frac{d^2(x_2,x_1)}{4t}\Big(1+\frac{kt\coth (kt)-1}{kt}\Big)
\Bigg].
\]
This completes the proof.
\end{proof}

As a direct corollary of Theorem \ref{lowerbound heat kernel thm1
intro}, we obtain the following estimate for $H(x,x,t)$.
\begin{coro}\label{lowerbound heat kernel thm1 cor1}
Under the same assumption as in Theorem \ref{lowerbound heat kernel
thm1 intro}, we have
\[
\begarr H(x,x,t)\ge& (4\pi
t)^{-\frac{n}{2}}2^{-\frac{n}{4}}\frac{(2kt)^\frac{n}{2}}{(e^{2kt}-2kt-1)^\frac{n}{4}},\\
{ and}&\\
H(x,x,t)\ge& (4\pi t)^{-\frac{n}{2}}\exp\Big[-\frac{n}{4}kt\Big].

\endarr
\]
\end{coro}

\begin{rema}
Results in this section hold for compact Riemannian manifold
with or without boundary condition. If the boundary is nonempty, we
need to assume the boundary is convex.
\end{rema}
%%%%%%%%%%%%%%%%%%%%%%
%%%%%%%%%%%%%%%%%%%%%%
%%%%%%%%%%%%%%%%%%%%%%
%%%%%%%%%%%%%%%%%%%%%%
\section{\bf A Perelman type differential Harnack inequality}\label{sec LYHP2}
We will devote this section to the proof of Theorem \ref{thm1.2}. Theorem \ref{thm1.2} may be the closet one to the differential Harnack inequality discovered by Perelman along Ricci flow. It's worthwhile to note that in this section we will follow a different notation which was used before by Yau in \cite{Yau2} and Perelman in \cite{P1}. Namely, we always assume $u=e^{-f}$ for positive heat solutions and $u=\frac{e^{-f}}{(4\pi t)^\frac{n}{2}}$ for positive heat kernels.  We start with a lemma. Let
\[
\begarr
\mathcal X(f,t)=&t\Delta f+f+\varphi(t),\quad {\rm with}\quad \varphi(t)=-\frac{n}{2}\ln (4\pi t)-n(1+\frac{1}{2}kt)^2\\

\mathcal Y(f,t)=&f_t=\Delta f-|\nabla f|^2,\\

 \mathcal Z(f,t)=&\mathcal X+t(1+kt)\mathcal Y.
\endarr
\]
Easy to see, $\int_M\mathcal Yudv=\int_Mf_tud\mu=\int_M(\Delta f-|\nabla f|^2)udv=0$. Hence, a multiple of $\int_M\mathcal Yudv$ with a time function will not affect the entropy functional.

\begin{lemm}\lab{lemm2.1}
Let $M$ be a complete Riemannian manifold. Let $u$ be a positive solution to the heat equation with $u=e^{-f}$. Then
\begin{enumerate}
  \item $(\frac{\partial}{\partial t}-\Delta)\mathcal X=-2t|\nabla_i\nabla_j f-(\frac{1}{2t}+\frac{k}{2})g_{ij}|^2-2\nabla \mathcal X\nabla f-(2kt+1)(\Delta f-|\nabla f|^2)-2t(R_{ij}+kg_{ij})f_if_j$\lab{lemm2.1-1}
  \item $(\frac{\partial}{\partial t}-\Delta)\mathcal Y=-2\nabla \mathcal Y\nabla f$\lab{lemm2.1-2}
  \item $(\frac{\partial}{\partial t}-\Delta)\mathcal Z=-2t|\nabla_i\nabla_j f-(\frac{1}{2t}+\frac{k}{2})g_{ij}|^2-2\nabla \mathcal Z\nabla f-2t(R_{ij}+kg_{ij})f_if_j$\lab{lemm2.1-3}
\end{enumerate}
\end{lemm}
\begin{proof}
By direct computations with the help of Bochner formula, (c.f. Lemma $2.2$ in \cite{NL1}), one can prove equality (\ref{lemm2.1-1}),
\[
\begarr
&(\frac{\partial}{\partial t}-\Delta)\mathcal X\\
\\
=& -2t\Delta|\nabla f|^2+(\Delta f-|\nabla f|^2)+\varphi'\\
\\
=& -2t|f_{ij}|^2-2t\nabla\Delta f\nabla f-2tR_{ij}f_if_j+(\Delta f-|\nabla f|^2)+\varphi'\\
\\
=& -2t|f_{ij}|^2-2\nabla\mathcal X\nabla f+2|\nabla f|^2+2kt|\nabla f|^2+(\Delta f-|\nabla f|^2)-2t(R_{ij}+kg_{ij})f_if_j+\varphi'\\
\\
=& -2t|f_{ij}-(\frac{1}{2t}+\frac{k}{2})g_{ij}|^2-2\nabla\mathcal X\nabla f-(2kt+1)(\Delta f-|\nabla f|^2)-2t(R_{ij}+kg_{ij})f_if_j.
\endarr
\]
Equality (\ref{lemm2.1-2}) is again by Bochner formula and it has already been proved in section \ref{sec gradient} Lemma \ref{sLY lem1}. Inequality (\ref{lemm2.1-3}) is immediate from equality (\ref{lemm2.1-1}) and equality (\ref{lemm2.1-2}),
\[
\begarr
(\frac{\partial}{\partial t}-\Delta)\mathcal Z=&(\frac{\partial}{\partial t}-\Delta)\mathcal X+t(1+kt)(\frac{\partial}{\partial t}-\Delta)\mathcal Y+(1+2kt)\mathcal Y\\
\\
=&-2t|\nabla_i\nabla_j f-(\frac{1}{2t}+\frac{k}{2})g_{ij}|^2-2\nabla \mathcal Z\nabla f-2t(R_{ij}+kg_{ij})f_if_j
\endarr
\]
\end{proof}

Now, we recall Theorem \ref{thm1.2}.\\

\noindent{\em {\bf Theorem \ref{thm1.2}}
Suppose $M^n$ is a closed manifold. Let $u$ be the positive heat kernel and $k\ge0$ is any constant satisfying $R_{ij}(x)\ge -kg_{ij}$ for all $x\in M^n$, then
\begeq\label{}
\begin{array}{rll}
v:=\big[t\Delta f+t(1+kt)(\Delta f -|\nabla f|^2)+f-n(1+\frac{1}{2}kt)^2\big]u\le 0,
\end{array}
\endeq
for all $t>0$ with $u=\frac{e^{-f}}{(4\pi t)^\frac{n}{2}}$.\\
}

We first derive the evolution equation equation of $v$.
\begin{prop}\label{prop2.1}
Let $u$, $f$, and $v$ be defined as in Theorem \ref{thm1.2}. Then
\begeq
(\frac{\partial}{\partial t}-\Delta)v=-2t\big|\nabla_i\nabla_j f-(\di\frac{1}{2t}+\frac{k}{2})g_{ij}\big|^2u-2t(R_{ij}+kg_{ij})f_if_ju.
\endeq
\end{prop}
\begin{proof}
Recall in Theorem \ref{thm1.2}, $u=\frac{e^{-f}}{(4\pi t)^\frac{n}{2}}$. One can use the change of variable idea, or simply observe that $f+\frac{n}{2}\ln(4\pi t)$ satisfies the assumption of Lemma \ref{lemm2.1}. Easy to see $v(f,t)=Z(f+\frac{n}{2}\ln(4\pi t),t)u$. Hence, by Lemma \ref{lemm2.1}
\earr
(\frac{\partial}{\partial t}-\Delta)v=&(\frac{\partial}{\partial t}-\Delta)(\mathcal Zu)\\
\\
=&(\frac{\partial}{\partial t}-\Delta)\mathcal Zu-2\nabla \mathcal Z\nabla u+\mathcal Z(\frac{\partial}{\partial t}-\Delta)u\\
\\
=&(\frac{\partial}{\partial t}-\Delta)\mathcal Zu+2u\nabla \mathcal Z\nabla f\\
\\
=&-2t\big|\nabla_i\nabla_j f-(\frac{1}{2t}+\frac{k}{2})g_{ij}\big|^2u-2t(R_{ij}+kg_{ij})f_if_ju,
\eearr
where we have used the fact that $(f+\frac{n}{2}\ln(4\pi t))_i=f_i$.
\end{proof}

The way we find the point-wise differential Harnack quantity can be used to find similar quantities for other geometric evolution equations, e.g., along Ricci flow equation. See an application in \cite{LX2}. \\

One can rewrite $v$ as $v=\big[t(2\Delta f -|\nabla f|^2)+f-n\big]u +kt^2\Delta u-nkt(1-\frac{1}{4}t)u$. Clearly, the first term is the one discussed by Ni for the $Ric\ge0$ case and the second term is divergence free. The last term contributes to the Ricci curvature term from Bochner formula. On the other hand, the difference between $v$ and the integrand of $\mathcal W_P$, cf. (\ref{new entropy}), is a divergence term $(t+kt^2)(\Delta f -|\nabla f|^2)u=(t+kt^2)\Delta u$. This term is crucial for finding the pointwise differential inequality. The evolution of $v$ also yields a proof to the monotonicity of $\mathcal W_P$ in section \ref{sec entropy}.\\

When $k=0$, our theorem reduces to the result of L. Ni in \cite{NL1} and \cite{NL4}. The reason we discuss the case of $k\ge0$ is because we will make essential use of the heat kernel comparison with hyperbolic space and Euclidean space, see Theorem A.3. in the Appendix. Going through more complicated computations, one could deal with $Ricci(M)>0$ case as well. Using ideas from this section, we also established a direct proof for Perelman's differential Harnack inequality along Ricci flow in \cite{LX2}.\\

\begin{rema}\label{trick}
Our proof of this differential inequality is different from the one of L. Ni for $Ric\ge0$ case in \cite{NL1} and \cite{NL4}. The main simplification is we do not need to use the gradient estimates for all positive solutions to the heat equation, and also other techniques such as the {\em nontrivial} reduced distance function introduced by Perelman in the study of Ricci flow.
\end{rema}

%One of the key ingredients of the proof is that we adapt an upper bound for derivatives of heat kernel by Stroock and Turetsky \cite{ST2}, which is different from the gradient estimates for positive solutions of the heat equation used in \cite{NL4}. Another important fact we need is the lower bound estimates of heat kernel by Cheeger-Yau's heat kernel comparison theorem \cite{CY}, together with estimates of Davies and Mandouvalos \cite{DM} on the heat kernel of space form. These observations enable us to finish the proof without appealing to the {\em nontrivial} Perelman reduced distance function (c.f. \cite{NL4}).\\

%The rest of this section will be devoted to prove Theorem \ref{thm1.2}. We shall prove the theorem for a complete manifold with Ricci curvature bounded from below following the idea of Perelman in \cite{P1}. The case of $Ric\ge0$ was proved by L. Ni \cite{NL1}. The key step is to justify the {\em claim} $\lim_{t\rightarrow 0}\int_Mvhd\mu= 0$. Our method to obtain the {\em claim} is different from the proof of L. Ni  in \cite{NL1} and \cite{NL4}.\\

\begin{proof}{\bf (Proof of Theorem \ref{thm1.2}:)}  For any $t_0 > 0$, let $h$ be any positive function. We solve the backward heat equation $(\partial_t-\Delta)h(y,t_0-t)=0$ starting from $t_0$ with initial data $h$. We then have that,

\begin{eqnarray}\label{monotone}
\partial_t\int_M vh d\mu &=&\int_M (h_tv+hv_t) d\mu \nonumber\\
&=&\int_M (-\Delta hv+h\Delta v)d\mu-2t\int_M\big|\nabla_i\nabla_j f-(\frac{1}{2t}+\frac{k}{2})g_{ij}\big|^2hud\mu\\
&& -2t\Dint_M(R_{ij}+kg_{ij})f_if_juhd\mu \nonumber\\
&=& -2t\int_M\big|\nabla_i\nabla_j f-(\frac{1}{2t}+\frac{k}{2})g_{ij}\big|^2hud\mu -2t\Dint_M(R_{ij}+kg_{ij})f_if_juhd\mu\nonumber\\
&\le&0,\nonumber
\end{eqnarray}
where we have used Proposition \ref{prop2.1} and $h_t+\Delta h=0$.

$${ {\emph{\bf Claim}}:} \quad\quad \lim_{t\rightarrow 0}\int_M vh d\mu\leq 0 \hspace{3cm}$$

Combining (\ref{monotone}) and the {\bf\emph{claim}}, we have that
$$\int_M h v d\mu\leq 0$$
for any $t_0>0$ and any positive functions $h$. This implies that $v\leq 0$.

\end{proof}

The key point is to prove the {\emph{claim}}. We need an upper bound on derivatives of the logarithm of the heat kernel and  the small time asymptotic estimates on logarithmic derivatives of the heat kernel. These known results are summarized in the Appendix.

\begin{proof}{\bf (Proof of the {\em Claim}:)} Define
$${\mathcal W_{h}}(t,x)= \int_M h(t_0-t,y) v(t,x,y) d\mu.$$
The proof is along the line of \cite{NL4}. We first show that for any fixed $x\in M$, ${\mathcal W_{h}}(t)$ has a finite upper bound as $t\rightarrow 0$. For the fundamental solution $F(t,x,y)=e^{-f}$ , $v=\big[t\Delta f+t(1+kt)(\Delta f -|\nabla f|^2)+f-\frac{n}{2}\ln (4\pi t)-n(1+\frac{1}{2}kt)^2\big]F$. We can write
\begin{eqnarray}
{\mathcal W_{h}}(t)&=& t(2+kt)\int_M \Delta f Fhd\mu- t(1+kt)\int_M|\nabla f|^2Fh d\mu+\int_M \big(f-\frac{n}{2}\ln 4\pi t\big) F h d\mu \nonumber\\ &&-n(1+\frac{1}{2}kt)^2\int_M F h d\mu \nonumber\\
&=& I+ II +III +IV\nonumber
\end{eqnarray}
By Theorem {\bf A2} (3) in the Appendix, we have
\begin{eqnarray}
\Bigg|\frac{\nabla F(t,x,y)}{F(t,x,y)}\Bigg|\leq D_1\Big(\frac{dist(x,y)}{t}+\frac{1}{\sqrt{t}}\Big) \nonumber\\
\Bigg|\frac{\Delta F(t,x,y)}{F(t,x,y)}\Bigg|\leq D_2\Big(\frac{dist(x,y)}{t}+\frac{1}{\sqrt{t}}\Big)^2 \nonumber
\end{eqnarray}
Since $\Delta f=-\frac{\Delta F}{F}+\frac{|\nabla F|^2}{F^2}$, we have
\earr
|I|=&\int_M \Big|t(2+kt)\Big(-\frac{\Delta F}{F}+\frac{|\nabla F|^2}{F^2}\Big)\Big|F hd\mu \leq D|t(2+kt)|\int_M \Big(\frac{dist^2(x,y)}{t^2}+\frac{1}{t}\Big) F h d\mu\\
\\
|II|=&\Big|t(1+kt)\Big|\int_M \frac{|\nabla F|^2}{F^2}F hd\mu\leq D|t(1+kt)|\int_M \Big(\frac{dist^2(x,y)}{t^2}+\frac{1}{t}\Big) F h d\mu\nonumber\\
\eearr

From the asymptotic expansion of the heat hernel of $F(t,x,y)$, and elementary computations, we have
\begin{eqnarray}
%&&\lim_{t\rightarrow 0}\beta(t)/t=1 \nonumber\\
&&\lim_{t\rightarrow 0}\int_M F(t,x,y) h(y) d\mu=h(0,x) \nonumber\\
&&\lim_{t\rightarrow 0}\int_M \frac{dist^2(x,y)}{4t}F(t,x,y) h(y) d\mu=\frac{n}{2}h(0,x) \label{distancesquare}\\
&&0\leq\lim_{t\rightarrow 0}\int_M  dist(x,y)F(t,x,y) h(y) d\mu \nonumber\\
&&\leq\lim_{t\rightarrow 0}\Big(\int_M \frac{dist^2(x,y)}{4t}F(t,x,y) h(y) d\mu\Big)^{1/2}\Big(t\int_MFh d\mu\Big)^{1/2} =0 \nonumber
\end{eqnarray}
where the second limit is by elementary computations on (\ref{Asymptotic}), (c.f. pg 8 in \cite{NL4}). Hence, we have $\limsup_{t\rightarrow 0}|I|+|II|\leq \tilde Dh(0,x)$.

When $k>0$, from Cheeger and Yau$'$s heat kernel comparison theorem for $Ric\ge-k$ and Davies and Mandouvalos$'$s lower bound estimate on heat kernel of space form, we have
\begin{eqnarray}
F(t,x,y)&\geq& C(n) (4\pi t)^{-n/2}\exp\Big(-\frac{r^2}{4t}-\frac{(n-1)kt}{4}-\frac{r\sqrt{(n-1)k}}{2}\Big)\nonumber\\
&&\times\Bigg(1+r\sqrt{\frac{k}{n-1}}+\frac{k}{n-1}t\Bigg)^{\frac{n-1}{2}-1}\Big(1+r\sqrt{\frac{k}{n-1} }\Big) \nonumber
\end{eqnarray}
where $r=dist(x,y)$. From $F=e^{-f}$, we have
\begin{eqnarray}
f-\frac{n}{2}\ln4\pi t&\leq& \frac{r^2}{4t}+\frac{(n-1)kt}{4}+\frac{r\sqrt{(n-1)k}}{2}-\ln C(n)\nonumber\\
&-&(\frac{n-1}{2}-1)\ln\Big(1+r\sqrt{\frac{k}{n-1}}+\frac{k}{n-1}t\Big)-\ln\Big(1+r\sqrt{\frac{k}{n-1} }\Big) \nonumber
\end{eqnarray}
Hence, from (\ref{distancesquare}), we have
\begin{eqnarray}
\limsup_{t\rightarrow 0}III =\limsup_{t\rightarrow 0}\int_M \big(f-\frac{n}{2}\ln(4\pi t)\big)Fh d\mu \leq \Big(\frac{n}{2}-\ln C(n)\Big)h(0,x) \nonumber
\end{eqnarray}
Similar arguments works for the case of $k=0$, where one need to deal with heat kernel of Euclidean space which has simpler form. On the other hand, $\limsup_{t\rightarrow 0}IV=-n h(0,x)$. Hence, we can conclude that $\limsup_{t\rightarrow 0}{\mathcal W_{h}}(t)$ is finite.\\

By the entropy monotonicity formula (\ref{monotone}), we know that the limit $\lim_{t\rightarrow 0}{\mathcal W_{h}}(t)=\gamma$ exists for some finite $\gamma$. Hence $\lim_{t\rightarrow 0}\Big[{\mathcal W_{h}}(t)-{\mathcal W_{h}}(\frac{t}{2})\Big]=0$. By (\ref{monotone}) and the mean-value theorem, we can find $t_i \rightarrow 0$ such that
\begin{eqnarray}
\lim_{t_i\rightarrow 0}t^2_i\int_M \Big|\nabla\nabla f-(\frac{1}{2t}+\frac{k}{2})g\Big|^2Fh d\mu =0 \nonumber
\end{eqnarray}
By the Cauchy-Schwartz inequality and the H\"older inequality, we have that
\[\begarr
&\lim_{t_i\rightarrow 0}t_i\Dint_M \Big(\Delta f-\frac{n}{2t_i}-\frac{1}{2}nk\Big)Fh d\mu \\
\\
=&t_i\lim_{t_i\rightarrow 0}\Dint_M \Big(\Delta f-\frac{n}{2t_i}\Big)Fh d\mu-\lim_{t_i\rightarrow 0}\frac{1}{2}nkt_i\int_M Fh d\mu \nonumber\\
\\
=&t_i\lim_{t_i\rightarrow 0}\Dint_M \Big(\Delta f-\frac{n}{2t_i}\Big)Fh d\mu =0.
\endarr\]
This yields
\[\begarr
\gamma=&\di\lim_{t_i\rightarrow 0}{\mathcal W_{h}}(t_i)\\
=&\di t_i(1-kt_i)\int_M  (\Delta f-\big|\nabla f|^2)Fh d\mu+\int_M\Big(f-\frac{n}{2}\ln(4\pi t_i)+\frac{n}{2}-n(1+\frac{1}{2}kt_i)^2\Big)Fhd\mu
\endarr\]

Using integration by parts, we have
\begin{eqnarray}
\int_M  (\Delta f-\Big|\nabla f|^2)Fh d\mu &=&-\int_M  \Delta F h d\mu  \nonumber\\
&=&-\int_M F\Delta h d\mu=-\Delta h(0,x) \nonumber
\end{eqnarray}
Hence we have
\begin{eqnarray}
\lim_{t_i\rightarrow 0}t_i(1+kt_i)\int_M \Big(\Delta f-\Big|\nabla f\Big|^2\Big)Fh d\mu =-\lim_{t_i\rightarrow 0}t_i(1+kt_i)\Delta h(0,x) =0 \nonumber
\end{eqnarray}
From (\ref{Asymptotic}) in the Appendix, we have
\begin{eqnarray}
\lim_{t_i\rightarrow 0}\Big(f-\frac{n}{2}\ln (4\pi t_i)-\frac{dist^2(x,y)}{4t_i}\Big)
&=&-\lim_{t_i\rightarrow 0}\ln\Big((4\pi t_i)^{n/2}e^{\frac{dist^2(x,y)}{4t_i}}F(t_i,x,y)\Big)\nonumber\\
&=&-\ln H_0(x,y)\nonumber
%\\&\leq&-\frac{n-1}{2}\ln \frac{\sinh\Big(\sqrt{\frac{k}{n-1}}dist(x,y)\Big)}{\sqrt{\frac{k}{n-1}}dist(x,y)}\leq 0 \nonumber
\end{eqnarray}
holds uniformly  on any compact subsets of $M \setminus Cut(x)$, and for $y = \exp_x(Y)$, $H_0(x, y)$ is given by the reciprocal of the square root of the Jacobian of $\exp_x$ at $Y$, and $H_0(x,x)=1$.
%and the last inequality comes from Bishop Theorem for Jacobian of $\exp_x$ (Theorem 1.2 in \cite{SY}).
Hence we have
\begin{eqnarray}
\gamma&=&\lim_{t_i\rightarrow 0}\int_M\Big(f-\frac{n}{2}\ln(4\pi t_i)+\frac{n}{2}-n(1+\frac{1}{2}kt_i)^2\Big)Fhd\mu \nonumber\\
&=&\lim_{t_i\rightarrow 0}\int_{M} \Big(f-\frac{n}{2}\ln (4\pi t_i)-\frac{dist^2(x,y)}{4t_i}\Big)Fh d\mu \nonumber\\
&&+\lim_{t_i\rightarrow 0}\int_{M} \Big(\frac{dist^2(x,y)}{4t_i}-\frac{n}{2}\Big)Fh d\mu +n(kt_i+\frac{1}{4}k^2t^2_i)\lim_{t_i\rightarrow 0}\int_{M}Fh d\mu \nonumber
\nonumber\\
&\le& -\big(\ln H_0(x,x)\big)h(0,x)=0 \nonumber
\end{eqnarray}
The last inequality comes from uniformly convergence theorem with $Cut(x)$ zero measure for first term, and (\ref{distancesquare}).

Hence we prove $\gamma\le 0$ holds, which is our Claim.

\end{proof}

\section{\bf Another Harnack Inequality for Heat Kernels}\label{sec LYHP Har}
Combining differential Harnack inequality (\ref{harnack inequ}) with the heat equation $u_t=\Delta u$, $u=\frac{e^{-f}}{(4\pi t)^\frac{n}{2}}$, and $|\nabla f|^2-\Delta f +f_t +\frac{n}{2t}=0$ we have a Hamilton-Jacobi inequality
\begeq
|\nabla f|^2 + (2+kt)f_t+\frac{f}{t}\le k\frac{n}{4}(2+kt).
\endeq
The case of $k=0$ is due to L. Ni \cite{NL3}. Naturally, the differential Harnack inequality (\ref{harnack inequ}) will leads to the following Harnack type estimates. 

\begin{theo}\label{LYH harnack}
Let $M^n$ be a Riemannian manifold with $Ric\ge -kg$ ($k\ge0$). If we denote $\tilde t:=\tilde t(t)=\frac{t}{2+kt}$, then for any $x_1$, $x_2\in M^n$, and $0<t_1<t_2$,
the following Harnack type estimates hold,
\begeq
\sqrt {\tilde t_2} f(x_2,t_2)-\sqrt {\tilde t_1} f(x_1,t_1)\le \frac{dist^2(x_1,x_2)}{4(s(t_2)-s(t_1))}+\frac{n}{4}(\Phi(t_2)-\Phi(t_1)),
\endeq
where $\tilde{t}_1=\tilde t(t_1)$, $\tilde{t}_2=\tilde t(t_2)$, \\
$$\Phi(t)=k\displaystyle\int^t_0\sqrt{\tilde t}dt=k^{-\frac{1}{2}}[\sqrt{kt(kt+2)}-\ln(1+kt+\sqrt{kt(kt+2)})],$$
 and $s=s(t)$ is defined as following
\begin{enumerate}
%  \item $s(t):=\sqrt{\frac{2}{k}}\arcsin(kt-1)$, for $k>0$;
  \item $s(t):=\displaystyle\int^t_0\frac{1}{\sqrt{\tilde t}}dt=k^{-\frac{1}{2}}[\sqrt{kt(kt+2)}+\ln(1+kt+\sqrt{kt(kt+2)})]$, for $k>0$;
  \item $s(t):=\sqrt t$, for $k=0$.
\end{enumerate}

\end{theo}

In this section, we always assume $(M^n,g)$ a complete (possibly noncompact) manifold with Ricci curvature bounded from below, i.e. $Ric(M)\geq -k$ for some constant $k\ge0$.  We shall apply an equivalent form of our differential Harnack inequality (\ref{harnack inequ}):
\begin{equation}
t(2+kt)\Delta f -t(1+kt)|\nabla f|^2+f-n(1+\frac{1}{2}kt)^2\le 0 \nonumber
\end{equation}
for the heat kernel $F(t,x,y)=\frac{e^{-f}}{(4\pi t)^\frac{n}{2}}$ on $M$ to obtain a pointwise Harnack inequality. The estimate (\ref{harnack inequ}) is a Li-Yau-Hamilton type since combining with the heat equation
$$f_t=\Delta f-|\nabla f|^2 -\frac{n}{2t},$$
we have a Hamilton-Jacobi inequality
\begin{eqnarray}
|\nabla f|^2 + (2+kt)f_t+\frac{f}{t}\le k\frac{n}{4}(2+kt). \label{HJI}
\end{eqnarray}

\begin{proof}({\bf Proof of Theorem \ref{LYH harnack}}:)
We only prove the case of $k>0$ here. Case $k=0$ is due to Ni \cite{NL3}. Let $\tilde t(t)$ and $s(t)$ be defined as in the theorem.
\begin{eqnarray}
\sqrt{\tilde t_2}f(t_2,x_2,y)-\sqrt{\tilde t_1}f(t_1,x_1,y)&=&\int_{t_1}^{t_2}\frac{d}{dt}\Big(\sqrt{\tilde t}f(t,\gamma(t),y)\Big)dt\nonumber\\
&=&\int_{t_1}^{t_2} \sqrt{\tilde t}\Big(f_t+\frac{(\sqrt{\tilde t})'}{\sqrt{\tilde t}}f+\<\nabla f,\gamma'(t)\>\Big)dt\nonumber\\
&\leq& \int_{t_1}^{t_2} \sqrt{\tilde t}\Big(-\frac{1}{2+kt}|\nabla f|^2+|\nabla f|\cdot|\gamma'|+\frac{nk}{4}\Big)dt\nonumber\\
&\leq&\int_{t_1}^{t_2} \Big(\frac{1}{4}\sqrt{t(2+kt)}|\gamma'(t)|^2+\frac{1}{4}nk\sqrt{\tilde t}\Big)dt\nonumber
\end{eqnarray}
for any path $\gamma(t)$ joining from $x_1$ to $x_2$. One can check that $s'(t)=\frac{1}{\sqrt{\tilde t}}$. This gives a Harnack type estimate:
\begin{eqnarray}
\sqrt{\tilde t_2}f(t_2,x_2,y)-\sqrt{\tilde t_1}f(t_1,x_1,y)&\leq& \inf_{\gamma}\frac{1}{4}\int_{s(t_1)}^{s(t_2)} |\gamma'(s)|^2ds +\frac{n}{4}(\Phi(t_2)-\Phi(t_1)) \nonumber
\end{eqnarray}
Choose $\gamma(s)$ to be a shortest geodesic with constant speed completes the proof.
\end{proof}

\begin{rema}\label{rem4.1}
The $k=0$ case was obtained by L. Ni in \cite{NL3}. The heat kernel comparison theorem of Cheeger and Yau \cite{CY} will imply the differential Harnack inequality as we did in proof of Theorem \ref{thm1.2}. Reversely, when the manifold is Ricci nonnegative, Ni recovered Cheeger-Yau's heat kernel comparison theorem by the above Harnack inequality. We thank Lei Ni for helping us understand his paper and pointing out a mistake in the early manuscript.
\end{rema}

\vspace{3mm}

Along the line of Ni, we consider the case of $k>0$. We take $x_1=o$ in the above theorem, where $o$ is the singular point of the fundamental solution at $t=0$. Argue as in \cite{NL3}, one gets $\displaystyle\lim_{t\rightarrow0}\sqrt{\tilde t}f(t,o,o)-\frac{n}{4}\Phi(t)\le0$, since $\displaystyle\lim_{t\rightarrow0}u(o,t)=\lim_{t\rightarrow0}\frac{\frac{\sqrt{\tilde t}f(o,t)}{\sqrt{\tilde t}}}{(4\pi t)^\frac{n}{2}}=\delta_0(o)$. This yields the following heat kernel comparison theorem.

\begin{coro}\label{Coro4.1}
Let $M^n$, $F$, $f$, $\tilde t$, $s$, and $\Phi$ be the same as in Theorem \ref{LYH harnack}. Then for any $(x,y)\in M\times M$, we have
\begin{enumerate}
  \item When $Ricci(M)\ge0$, $f(x,t)\le\frac{d^2(x,o)}{4t}$.
  \item When $Ricci(M)\ge-k$, $f(x,t)\le\frac{d^2(x,o)}{4\sqrt{\tilde t}s(t)}+\frac{n}{4}\frac{\Phi(t)}{\sqrt{\tilde t}}$.
\end{enumerate}
\end{coro}

\section{\bf Entropy formulas with monotonicity}\label{sec entropy}
In this section, we will introduce various new entropy functionals and discuss their monotonicity along the heat equation on any compact Riemannian manifold. Point-wise differential Harnack inequalities and monotonicity formulas for entropy functionals are closely related. Usually, a point-wise differential Harnack quantity easily yields a monotonicity formula for the related functional. But reversely, it is more difficult. In general, the proofs of monotonicity formulas for functionals are also easier. The reason is upon integration over closed manifolds, all the information of a divergence form will be disappeared. This implies the point-wise differential Harnack quantity should be {\em the} representative in the space of the entropy integrands for the same entropy functional.\\

In this section, based on the Li-Yau type and Li-Yau-Hamilton-Perelman type of differential Harnack inequalities
we introduced in section \ref{sec gradient} and \ref{sec LYHP2}, we can easily establish monotonicity formulas for the  related entropy functionals. But in our actual searching for differential Harnack, we discovered the functionals first, then localized them and obtained the pointwise version.\\

%In order to make a comparison with the entropies introduced in Addenda of \cite{NL4}, we will intentionally introduced the entropies in similar order.

As in section \ref{sec LYHP2}, we will follow the notations of Yau \cite{Yau2} and Perelman \cite{P1} to assume $u=e^{-f}$ for positive heat solutions and $u=\frac{e^{-f}}{(4\pi\tau)^\frac{n}{2}}$ for positive heat kernels. \\

We introduce the following Li-Yau entropy functional $\mathcal W_{LY}$, where the integrand is $t^2Fu$ or $\sinh^2(kt)Fu$ from Proposition \ref{entropy prop1},
\begin{equation}\label{Li-Yau entropy}
\begarr
{\mathcal W_{LY}}(u,t):=&\Dint_{M^n}t^2Fud\mu=-\Dint_{M^n}t^2\Big[\Delta \ln u + \frac{n}{2t}+\frac{nk}{2}(1+\frac{1}{3}kt) \Big]ud\mu,\\
{\rm or}&\\
{\mathcal W_{LY}}(u,t):=&\Dint_{M^n}\sinh^2(kt)Fud\mu=-\Dint_{M^n}\sinh^2(t)\Big[\Delta \ln u + \frac{nk}{2}\big[\coth (kt)+1\big] \Big]ud\mu,
\endarr
\end{equation}
where we have used integration by parts to get $\int_Mf_tud\mu=\int_M(|\nabla \ln u|^2-\Delta\ln u)ud\mu=0$.\\

As a direct consequence of Proposition \ref{entropy prop1} and the differential Harnack inequality in Theorem \ref{main thm}, we have the following theorem.

\begin{theo}\label{Li-Yau entropy thm1}
Let $M^n$ be a closed manifold. Assume that $u$ is a positive solution to the heat equation (\ref{heat equ}) with $\int_Mud\mu=1$ and let $f=-\ln u$. Consider the functional $$\mathcal W_{LY}=-\Dint_{M^n}t^2\Big[\Delta \ln u + \frac{n}{2t}+\frac{nk}{2}(1+\frac{1}{3}kt) \Big]ud\mu.$$ If $Ricci(M)\ge -k$, then $\mathcal W_{LY}\le 0$ for all $t\ge 0$  and \begin{equation}
\frac{d}{dt}{\mathcal W_{LY}(f,t)}=-2t^2\Dint_{M^n}|\textstyle f_{ij}-\frac{k}{2}g_{ij}-\frac{1}{2t} g_{ij}|^2ud\mu-2t^2\Dint_{M^n}(R_{ij}+kg_{ij})f_if_jud\mu\leq 0.
\end{equation}
The monotonicity is strict for all $t\ge 0$, unless the manifold is Einstein, i.e. $Ricci(M)=-k$, and $f$ satisfies the gradient Ricci soliton equation $\textstyle\frac{1}{2}R_{ij}+f_{ij}-\frac{1}{2t} g_{ij}\equiv0$.
\end{theo}
\begin{rema}\label{Li-Yau entropy remark1}
On a complete noncompact manifold, the above Ricci soliton equation with potential function as the logarithmic of a positive heat solutionin can be realized in  some cases, e.g., heat kernels on $\mathbb R^n$ obtains the equality with $k=0$. 
\end{rema}

\begin{proof}
Recall that $f_t=\Delta f-|\nabla f|^2$. Using (\ref{linear thm equ3}), we have 
\[
\mathcal W_{LY}=\displaystyle\int_Mt^2Fud\mu\le0.
\] 

The monotonicity follows from $W_{LY}=\Dint_Mt^2Fud\mu$ and Proposition \ref{entropy prop1}.
\end{proof}

Exactly the similar theorem is also true for 
\[
{\mathcal W_{LY}}(u,t)=-\Dint_{M^n}\sinh^2(kt)\Big[\Delta \ln u + \frac{nk}{2}\big[\coth (kt)+1\big] \Big]ud\mu.
\]
We leave the details to the readers.\\

On the other hand, one can prove these theorems directly. Since the entropy integrand becomes simpler than the differential Harnack quantity and integration by parts works for closed manifolds, one can get an easier and more direct derivation for the monotonicity formula. We will use this idea to derive the Perelman type of entropy monotonicity in the proof of Theorem \ref{thm1.1}.\\

In regard of the Perelman type LYH Harnack quantity, the following entropy formula is very natural. We define
\begin{equation}\label{new entropy}
{\mathcal W_{P}}(f,\tau)=\Dint_{M^n}\Big(\tau|\nabla f|^2+f-n(1+\textstyle\frac{1}{2}k\tau)^2\Big)\frac{e^{-f}}{(4\pi\tau)^{\frac{n}{2}}}d\mu,
\end{equation}
with $\Dint_M\frac{e^{-f}}{(4\pi\tau)^{\frac{n}{2}}}d\mu=1$. When $k=0$, ${\mathcal W_P}$ is exactly Ni's functional $\mathcal W$ in \cite{NL1}. The following theorem generalizes L. Ni's result in the sense that for closed manifolds there is no curvature condition needed.

\begin{theo}\label{thm1.1}
Let $M^n$ be a closed manifold. Assume that $u$ is a positive solution to the heat equation (\ref{heat equ}) with $\int_Mud\mu=1$. If we choose $k\in {\mathbb R}$ to be any constant satisfying $Ricci(M)\ge -k$ and let $f$ be defined as $u=\frac{e^{-f}}{(4\pi\tau)^\frac{n}{2}}$ and $\tau=\tau(t)$ with $\frac{d\tau}{dt}=1$, then $W_P\le 0$ for all $t\ge 0$, and
\begin{equation}
\frac{d}{dt}{\mathcal W_P}=-2\tau\Dint_{M^n}|\textstyle\frac{k}{2}g_{ij}+f_{ij}-\frac{1}{2\tau} g_{ij}|^2\frac{e^{-f}}{(4\pi\tau)^\frac{n}{2}}d\mu-2\tau\Dint_{M^n}(R_{ij}+kg_{ij})f_if_j\frac{e^{-f}}{(4\pi\tau)^\frac{n}{2}}d\mu\leq 0.
\end{equation}
Moreover, the monotonicity is strict for all $t\ge 0$, unless the manifold is Einstein, $Ricci(M)=-k$, and $f$ satisfies the Ricci soliton equation $\textstyle\frac{1}{2}R_{ij}+f_{ij}-\frac{1}{2t} g_{ij}\equiv0$.
\end{theo}

There are various proofs for entropy monotonicity formula. An immediate proof is by using Proposition \ref{prop2.1}. Here, we will present a direct proof which is based on a change of variable argument, see similar argument in \cite{L1}.

\begin{proof}{\bf (Proof of Theorem \ref{thm1.1})}
 Let $\tilde f:=-\ln u = f+\frac{n}{2}\ln (4\pi \tau)$. Easy to see $\tilde f_t=\Delta \tilde f-|\nabla \tilde f|^2$. We first observe that the derivative of the Nash entropy $\mathcal N(f,\tau):=\Dint_M fe^{-f}d\mu $ is
\[
\begarr
\Dfrac{d}{dt}\mathcal N(\tilde f,\tau)= & \Dint_M\tilde f_tud\mu+\Dint_M\tilde fu_td\mu=\Dint_M(\Delta \tilde f-|\nabla \tilde f|^2)ud\mu+\Dint_M\tilde f\Delta ud\mu\\
\\
=& \Dint_M\Delta\tilde f ud\mu,
\endarr
\]
where the last step follows from integration by parts. As far as integration by parts is allowed, we can use the fact $\int_M|\nabla f|^2e^{-f}d\mu =\int_M\Delta fe^{-f}d\mu $.

Applying Bochner formula, one get,
\[
\begarr
\Dfrac{d}{dt}\Dint_M \tau|\nabla \tilde f|^2ud\mu=&\Dfrac{d}{dt}\Dint_M \tau\Delta \tilde fud\mu\\
\\
= &\Dint_M \Delta \tilde fud\mu +\tau\Dint_M (\Delta \tilde f_tu+\Delta \tilde fu_t)d\mu \\
%\\
%= &\Dint_M \Delta \tilde fud\mu +t\Dint_M (\tilde f_t+\Delta \tilde f)\Delta ud\mu\\
\\
= &\Dint_M \Delta \tilde fud\mu +\tau\Dint_M (2\Delta \tilde f-|\nabla \tilde f|^2)\Delta ud\mu\\
\\
= &\Dint_M \Delta \tilde fud\mu -2\tau\Dint_M (|\tilde f_{ij}|^2+R_{ij}\tilde f_i\tilde f_j)ud\mu\\
\\
= &\Dint_M \Delta \tilde fud\mu -2\tau\Dint_M |\tilde f_{ij}|^2ud\mu+2k\tau\Dint_M|\nabla \tilde f|^2ud\mu\\&
-2\tau\Dint_M(R_{ij}+kg_{ij})f_if_jud\mu
\endarr
\]
Now we are ready to obtain the monotonicity by using integration by parts and completing the square,
\[
\begarr
&\Dfrac{d}{dt}\Dint_M \Big(\tau|\nabla \tilde f|^2+\tilde f-\frac{n}{2}\ln (4\pi \tau)-n(1+\textstyle\frac{1}{2}k\tau)^2\Big)ud\mu\\
\\
= &2\Dint_M \Delta \tilde fud\mu -2\tau\Dint_M |\tilde f_{ij}|^2ud\mu+2k\tau \Dint_M|\nabla \tilde f|^2ud\mu -\frac{n}{2\tau}+nk-\frac{nk^2}{2}\tau\\
&-2\tau\Dint_M(R_{ij}+kg_{ij})f_if_jud\mu\\
\\
=&-2\tau\Dint_M |\tilde f_{ij}-(\frac{1}{2\tau}+\frac{k}{2})g_{ij}|^2ud\mu-2\tau\Dint_M(R_{ij}+kg_{ij})f_if_jud\mu
\endarr
\]
Change $\tilde f$ back to $f$, we complete the proof.\\
\end{proof}

The third interesting entropy functional is the `Nash entropy', $-\int_MH\log Hd\mu$, where $H$ is the positive heat kernel. We will use the linearized version of our generalized Li-Yau estimate, namely, the estimate in Theorem \ref{linear thm}, to illustrate the idea. The nonlinear version works exactly in the same way. 

 Following the ideas in Section 5 of \cite{P1} and motivated by and along the line of Addenda to \cite{NL1}, we discuss the relations among these different entropies. Let $u(x,t)$ be a positive solution to the heat equation with $\int_Mud\mu = 1$. We define
\[
N(u,t)= \Dint_M -(\log u)ud\mu
\]
and

\[
\tilde N(u,t)=N(u,t) -\frac{n}{2}\log (4\pi t)-\frac{n}{2}kt(1+\frac{1}{6}kt)-\frac{n}{2}.
\]
Direct computations shows that
\earr
\frac{d\tilde N}{dt}=&-\Dint_M \Big(\Delta (\log u)+\frac{n}{2t}+\frac{nk}{2}(1+\frac{1}{3}kt)\Big)ud\mu\\
=&\Dint_M \Big(|\nabla \log u|^2-(1+\frac{2}{3}kt)\frac{u_t}{u}-\frac{n}{2t}-\frac{nk}{2}(1+\frac{1}{3}kt)\Big)ud\mu,
\eearr
where in the last step we have used integration by parts and the heat equation.\\

Notice that the integrand in the last step is just the generalized Li-Yau gradient estimate (\ref{linear thm equ3}), which is
\begeq\label{generalized Li-Yau gradient}
|\nabla \log u|^2-(1+\frac{2}{3}kt)\frac{u_t}{u}-\frac{n}{2t}-\frac{nk}{2}(1+\frac{1}{3}kt)\le 0,
\endeq
for any closed manifold when choosing proper $k$. Using (\ref{generalized Li-Yau gradient}), one arrives at the following estimate on the Nash entropy $-\int_MH\log Hd\mu$, which extends L. Ni's result to general manifolds.
\begin{prop}
Let $M^n$ be a complete Riemannian manifold with $Ricci(M)\ge -k$ and $H$ be the positive heat kernel. Then $\tilde N(H,t)$ satisfies the following properties:
\begin{enumerate}
  \item $\frac{d}{dt}\tilde N<0$, unless $M$ is an Einstein manifold and $H$ satisfies the gradient Ricci soliton equation $\frac{1}{2}R_{ij}-(\ln u)_{ij}-\frac{1}{2\tau} g_{ij}\equiv 0$.
  \item $\lim_{t\rightarrow 0}\tilde N(H,t)=0$.
\end{enumerate}
\end{prop}
\begin{proof}
The monotonicity is a simple consequence of the generalized Li-Yau gradient estimates for heat kernels on complete manifolds with $Ricci(M)\ge -k$. Since the manifold may be noncompact, one can not use Theorem \ref{linear thm} directly. But for heat kernels, one can easily extend Theorem \ref{linear thm} to noncompact manifold by using the techniques we developed to prove the {\em claim} in Section \ref{sec LYHP2}. The equality case is from the vanishing of the first variation formula.\\
\end{proof}

The study of relations between pointwise differential Harnack inequality and monotonicity of entropy functionals for Ricci flow equations and heat equations is an important and very active field. As we have revealed in this paper, for both equations, Ricci soliton plays important role. See Entropy formulas for Ricci flow in Perelman's original work \cite{P1}, and others, e.g., \cite{FIN}, \cite{L1}.

\section{\bf Appendix}
We summarize some known results about heat kernels on manifolds in this Appendix.
\\

\noindent{\em {\bf Theorem A.1.} \label{Prop3.1}
Let $M^n$ be a smooth, complete Riemannian manifold. Let $C(M)\subset M\times M$ be the set of pairs of points $(x, y)$ such that $y \in Cut(x)$. Let $F(t,x,y)$ be the positive fundamental solution of the heat equation $\partial_t u(t,x)=\Delta u(t,x)$, define
$$E_t(x,y)=-2t\ln F(t,x,y),\hspace{1cm} E(x,y)=\frac{1}{2}{\rm dist}^2(x,y).$$
Then there are smooth functions $H_i(x, y)$ defined on $(M \times M)\setminus C(M)$ such that the asymptotic expansion
\begin{equation}\label{Asymptotic}
F(t,x,y)\sim \Big(\frac{1}{4\pi t}\Big)^{n/2}e^{-\frac{dist^2(x,y)}{4t}}\sum_{i=0}^{\infty}H_i(x,y) t^i
\end{equation}
holds uniformly as $t \rightarrow 0$ on compact subsets of $(M \times M)\setminus C(M)$. Further, if $y = \exp_x(Y )$, then $H_0(x, y)$ is given by the reciprocal of the square root of the Jacobian of $\exp_x$ at $Y$.

}

Furthermore, the following estimates on logarithmic derivatives of the heat kernel on $M$ are known:\\

\vspace{2mm}
\noindent{\em {\bf Theorem A.2.} \label{ABCD}
\begin{enumerate}
  \item {\bf(Varadhan \cite{V}, Cheng-Li-Yau \cite{CLY})} On any compact subsets of $M \times M$,
$$\lim_{t\rightarrow 0}E_t(x,y)=E(x,y)\; uniformly ;$$
  \item {\bf(Malliavin and Stroock \cite{MS}, Stroock and Turetsky \cite{ST1})}  On any compact subsets of $(M \times M)\setminus C(M)$,
$$\lim_{t\rightarrow 0}\nabla^m E_t(x,y)=\nabla^m E(x,y) \; uniformly;$$
  \item {\bf(Stroock and Turetsky \cite{ST2}, Hsu, Elton \cite{Hsu})} There are upper bounds for derivatives of the heat kernel on any closed manifold $M$ as
$$\big|\nabla^m F(t,x,y)\big|\leq D_m\Big(\frac{{\bf dist}(x,y)}{t}+\frac{1}{\sqrt{t}}\Big)^m F(t,x,y) $$
where the $D_m$ are some constants depending only on $M$.
  \item {\bf(Neel \cite{Ne})} For any $A \in T_y M$, we have

\begin{equation}\label{}
\left\{
\begin{array}{rll}
-|A| dist(x, y)&\leq \liminf_{t\rightarrow 0}\nabla_A E_t(x, y)\\
&\leq \limsup_{t\rightarrow 0}\nabla_A E_t(x, y)\leq |A| dist(x, y)\\
& and\\
-|A|^2 dist^2(x, y)&\leq \liminf_{t\rightarrow 0}t\nabla^2_{A,A} E_t(x, y)\\
&\leq \limsup_{t\rightarrow 0}t\nabla^2_{A,A} E_t(x, y)\leq 0
\end{array}
\right.\nonumber
\end{equation}
hold for any  $(x,y)\in M \times M$.
\end{enumerate}

}

\vspace{3mm}

For heat kernels on hyperbolic space, Davies and Mandouvalos have the following estimates.\\

\noindent{\em {\bf Theorem A.3. (Davies-Mandouvalos \cite{DM})}\label{DM}
Let $F^K(t,x,y)=F^K(t,{\bf dist}(x,y))$ be the heat kernel of $\Delta$ on $M^K$, the space form with constant sectional curvature $-K \leq 0$. Then
$$c(n)^{-1}h(t,{\bf dist}(x,y))\leq F^K(t,x,y) \leq c(n)h(t,{\bf dist}(x,y)) $$
where $c(n)$ depends only on dimension $n$ and
\begin{eqnarray}
h(t,r)&=&(4\pi t)^{-n/2}\exp\Big(-\frac{r^2}{4t}-\frac{(n-1)^2Kt}{4}-\frac{(n-1)\sqrt{K}r}{2}\Big)\label{bound-SF}\\
&&\times\Big(1+\sqrt{K}r+Kt\Big)^{\frac{n-1}{2}-1}(1+\sqrt{K}r). \nonumber
\end{eqnarray}
}

Another very useful estimate is the heat kernel comparison theorem of Cheeger and Yau.\\

\noindent{\em {\bf Theorem A.4. (Cheeger-Yau \cite{CY})}\label{cheeger-yau}
$$ F(t,x,y)\geq F^K(t,x,y),\; \forall\; Ric(M)\geq -(n-1)K.$$
We have the following lower bound estimates for heat kernel on  $M$:
\begin{eqnarray}
F(t,x,y)\geq c(n)^{-1}h(t,{\bf dist}(x,y)) \label{Lowerbound-HK}
\end{eqnarray}
where $h(t,r)$ as in (\ref{bound-SF}).
}

The following theorem is from \cite{SY} page $167$, Corollary 2.\\

\noindent{\em {\bf Theorem A.5. (Li-Yau \cite{SY})}\label{Li-yau cited corollary}
Let $H(x,y,t)$ be the heat kernel of a complete Riemannian manifold $M$. For any $\rho>0$, $T>0$, set
\begin{equation}
F(y,t)=\Dint_{M\setminus B_x(\rho)}H(x,\xi,T)H(\xi,y,t)d\xi.
\end{equation}
Then for any $\delta>0$, and $R>0$,
\earr
&\Dint_{B_x(R)}F^2(y,(1+\delta)T)dy\\
\le&\exp\Big( \frac{R^2}{2\delta T}\Big)\cdot \exp\Big(\frac{-\rho^2}{2(1+2\delta)T}\Big)\cdot\Dint_{M\setminus B_x(\rho)}H^2(x,\xi,T)d\xi.\\
\eearr
Moreover, if $\rho=0$, i.e.
\begin{equation}
F(y,t)=\Dint_{M}H(x,\xi,T)H(\xi,y,t)d\xi,
\end{equation}
then for any $\delta>0$, $T>0$, and $R>0$, we have
\earr
&\Dint_{B_x(R)}F^2(y,(1+\delta)T)dy\le\exp\Big( \frac{R^2}{2\delta T}\Big)F(x,T).\\
\eearr

}

\section{\bf Acknowledgement}
The authors would like to thank professor Pengfei Guan for encouragement and many informative discussions. Especially the present version of the local estimates was suggested and stimulated in his working seminar. Both authors thank CRM and McGill University for finacial support during the visit.

\end{document}